\documentclass[reqno]{amsart}
\usepackage[margin=1in]{geometry}
\usepackage{xcolor}
\usepackage{enumitem}
\usepackage{amssymb}
\usepackage{clipboard}
\usepackage{graphicx}
\usepackage{tikz}
\usepackage{academicons}
\definecolor{orcidlogocol}{HTML}{A6CE39}

\definecolor{bluecite}{HTML}{0875b7}

\usepackage[unicode=true,
bookmarksopen={true},
pdffitwindow=true,
colorlinks=true,
linkcolor=bluecite,
citecolor=bluecite,
urlcolor=bluecite,
hyperfootnotes=false,
pdfstartview={FitH},
pdfpagemode= UseNone]{hyperref}

\newcommand{\defeq}{\stackrel{\textup{def}}{=}}
\newcommand{\diff}{\,\mathrm{d}}

\DeclareMathOperator{\di}{div}

\newtheorem{theorem}{Theorem}[section]
\newtheorem*{theorem*}{Theorem}
\newtheorem{proposition}[theorem]{Proposition}

\theoremstyle{definition}

\newtheorem{remark}[theorem]{Remark}

\title[Sharp Hardy and spectral gap inequalities on special irreversible Finsler manifolds]{Sharp Hardy and spectral gap inequalities\\ on special irreversible Finsler manifolds}
\author{S\'andor Kaj\'ant\'o}\address{ \textsc{S\'andor Kaj\'ant\'o}: Department of Mathematics and Computer Science, Babe\c s-Bolyai University, Cluj-Napoca, Romania \& Institute
of Applied Mathematics, Óbuda University, Budapest, Hungary. ORCID: \href{https://orcid.org/0009-0005-4141-9734}{0009-0005-4141-9734}.}
\email{sandor.kajanto@ubbcluj.ro}
\date{\today}
\thanks{The author is supported by the János Bolyai Research Scholarship of the Hungarian Academy of Sciences.}


\begin{document}
\begin{abstract}
    The sharpness of various Hardy-type inequalities is well-understood in the reversible Finsler setting; 
    while infinite reversibility implies the failure of these functional inequalities, cf.~Kristály, Huang, and Zhao [\emph{Trans.\ Am.\ Math.\ Soc.}, 2020]. However, in the remaining case of irreversible manifolds with finite reversibility, there is no evidence on the sharpness of Hardy-type inequalities. In fact, we are not aware of any particular examples where the sharpness persists. In this paper we present two such examples involving two celebrated inequalities: the classical/weighted Hardy inequality (assuming non-positive flag curvature) and the McKean-type spectral gap estimate (assuming strong negative flag curvature). In both cases, we provide a family of Finsler metric measure manifolds on which these inequalities are sharp. We also establish some sufficient conditions, which guarantee the sharpness of more involved Hardy-type inequalities on these spaces. Our relevant technical tool is a Finslerian extension of the method of Riccati pairs (for proving Hardy inequalities), which also inspires the main ideas of our constructions. 
\end{abstract}

\vspace*{-1cm}
\maketitle
\vspace{-0.5cm}
\tableofcontents
\vspace{-0.7cm}

\section{Introduction} \label{sec:intro}
The celebrated \emph{Hardy inequality}, see Hardy~\cite{hardy1920note}, states that if \(u\) is a smooth, compactly supported function on \(\Omega\subseteq \mathbb{R}^n\), with \(n\ge3\), then the \(L^2\)-norm of the singular term \(u(x)/|x|\) is controlled by the \(L^2\)-norm of \(|\nabla u(x)|\). More precisely, one has
\begin{equation}\label{eq:intro:Hardy}
    \int_{\Omega} |\nabla u(x)|^2\diff x\geq \frac{(n-2)^2}{4}\int_{\Omega} \frac{|u(x)|^2}{|x|^2}\diff x,\qquad \forall u\in C_0^\infty(\Omega).
\end{equation}
Moreover, the above inequality is \emph{sharp}, in the sense that the constant \(\frac{(n-2)^2}{4}\) cannot be improved.

Due to the lack of an extremal function (for which equality holds), there exist various extensions and improvements of inequality~\eqref{eq:intro:Hardy}, which turned out to be the fundamental building blocks of various results concerning elliptic problems; for   comprehensive discussions, we refer to the monographs by Balinsky, Evans, and Lewis~\cite{balinsky2015analysis} and Ghoussoub and Moradifam~\cite{ghoussoub2013functional}. Over the years, the mathematical community has answered many questions regarding Hardy-type inequalities, but several still remained open. The present paper focuses on such an unanswered problem, which is the issue of sharpness on non-compact and non-reversible Finsler manifolds with finite reversibility. Before presenting our main results, let us enlist some relevant facts. 

The natural habitat for a Finslerian Hardy-type inequality is a \emph{Finsler metric measure manifold}~\((M,F,\mathsf{m})\) abbreviated by FMMM, which is a differentiable manifold~\(M\) equipped with a smooth, positive-homogeneous metric~\(F\) on the tangent bundle \(TM\), and a smooth positive measure~\(\mathsf{m}\).
When formulating Hardy-type inequalities in this setting, the expression \(|\nabla u (x)|\) is replaced by either \(\max\{F^*(\pm Du)\}\) or \(F^*(Du)\), where \(F^*\) is the dual metric of \(F\), and \(Du\) denotes the \emph{derivative} of \(u\) (see relations~\eqref{eq:prelim:Fdual}~\&~\eqref{eq:prelim:Du}). The quantity \(|x|\) is typically replaced by a \emph{distance function} \(\rho=d(x_0,\cdot)\) from a point \(x_0\in M\).

It turns out that three Finslerian quantities control the behavior of Hardy inequalities:
\begin{itemize}[leftmargin=0.8cm,topsep=0cm]
    \item The \emph{flag curvature}~\(\mathbf{K}\) (see relation~\eqref{eq:prelim:K}), which is the natural Finslerian extension of the sectional  curvature of Riemannian manifolds.
    \item The \emph{\(S\)-curvature}~\(\mathbf{S}\) (see relation~\eqref{eq:prelim:S}), which is a pure Finslerian quantity measuring the rate of change of \emph{distortion} along geodesics. For convenience, one can also use the \emph{reduced \(S\)-curvature} \(\overline{\mathbf{S}}=\mathbf{S}/F\).
    \item The \emph{reversibility}, which is defined by
          \[
              \lambda_F(M)\defeq\sup_{(x,y)\in TM} \mathbf{rev}(x,y), \quad\mbox{where}\quad \mathbf{rev}(x,y)=\frac{F(x,-y)}{F(x,y)}.
          \]
          Clearly \(1\le\lambda_F(M)\le +\infty\). The manifold is said to be \emph{reversible}, if \(\lambda_F(M)=1\). In particular, every Riemannian metric is reversible.
\end{itemize}

The Finslerian version of the Hardy inequality~\eqref{eq:intro:Hardy}, is due to Huang, Kristály, and Zhao~\cite[Theorem 1.4]{huang2020sharp}, where the authors showed that if \(\mathbf{K}\le 0\) and \(\mathbf{S}\le0\), then one has
\begin{equation}\label{eq:intro:Hardy:Finsler}
    \int_M \max\{F^*(\pm Du)\}^2\diff \mathsf{m}\ge \left(\frac{n-2}{2}\right)^2\int_M\frac{u^2}{\rho^2}\diff \mathsf{m},\quad \forall  u\in C_0^\infty(M);
\end{equation}
in addition, if \(F\) is reversible, then \( \left(\frac{n-2}{2}\right)^2\) is sharp. 
% We refer to Zhao~\cite[Theorem 1.1]{Z} for a weighted version of the above result. 
% \begin{equation}\label{eq:intro:lammax}
%     \lambda_F(M)F^*(Du)\ge \max\{F^*(\pm Du)\},
% \end{equation}
Since \(\lambda_F(M)F^*(Du)\ge \max\{F^*(\pm Du)\}\), 
% the above
% inequality %~\eqref{eq:intro:Hardy:Finsler}  
% implies
we also have
\begin{equation}\label{eq:intro:Hardy:Finsler:lam}
    \lambda_F(M)^2\int_M F^*(Du)^2\diff \mathsf{m}\ge \left(\frac{n-2}{2}\right)^2\int_M\frac{u^2}{\rho^2}\diff \mathsf{m},\quad \forall  u\in C_0^\infty(M).
\end{equation}

Note that if \(\lambda_F(M)=1\), then inequalities~\eqref{eq:intro:Hardy:Finsler}~\&~\eqref{eq:intro:Hardy:Finsler:lam} coincide. If \(\lambda_F(M)=+\infty\), then left hand side of the latter inequality blows up. In this case, the authors in~\cite[Example 6.2]{huang2020sharp} showed that concerning the \emph{Funk metric} with \(\mathbf{K}=-\frac{1}{2}\), \(\overline{\mathbf{S}}=\frac{n+1}{2}\), and \(\lambda_F=+\infty\) one has
\[
    \inf_{u\in C_0^\infty(M)}\frac{\int_M {F^*}^2(Du)\diff \mathsf{m}}{\int_M\frac{u^2}{\rho^2}\diff \mathsf{m}}=0,
\]
which aligns (in some sense) with inequality~\eqref{eq:intro:Hardy:Finsler:lam}.
We also refer to Kristály, Li, and, Zhao~\cite[Theorem 1.1]{kristaly2024influence} for additional results describing the failure of various functional inequalities.

In the remaining case when \(1<\lambda_F(M)<\infty\), to the best of our knowledge, there is no evidence on the sharpness of~\eqref{eq:intro:Hardy:Finsler}~\&~\eqref{eq:intro:Hardy:Finsler:lam}. In fact, we are not aware of any particular examples where the sharpness persists. Our first main result is an affirmative example, based on the construction of suitable spaces presented below. Since the sharpness of~\eqref{eq:intro:Hardy:Finsler} follows from the sharpness of \eqref{eq:intro:Hardy:Finsler:lam}, we only discuss the latter. 

Consider the triple \((\mathbb{R}^n,F,\mathsf{m})\), where \(F\colon \mathbb{R}^n\times \mathbb{R}^n\to [0,\infty)\) is defined by
\[
    F(x,y)=|y|-\frac{\langle x,y\rangle\theta}{|x|},\quad\text{if } x\ne O\quad\text{and}\quad F(O,y)=|y|,
\]
for some \(\theta\in(0,1)\) and \(\mathsf{m}\) stands for the Busemann--Hausdorff measure induced by \(F\) (see relation~\eqref{eq:prelim:BH}). Here \(O\in \mathbb{R}^n\) denotes the origin, \(|\cdot|\) and \(\langle\cdot,\cdot\rangle\) are the norm and the inner product in \(\mathbb{R}^n\), respectively. Note that \(F\) is \emph{not} a Finsler metric, since it is not continuous at the origin, but it has some remarkable properties, which could facilitate the proof of the sharpness of~\eqref{eq:intro:Hardy:Finsler:lam}: 
The geodesics are straight lines. The reversibility is \(\lambda_F(\mathbb{R}^n)=\frac{1+\theta}{1-\theta}=\lambda\in(1,\infty)\). If \(\rho=d(O,\cdot)\), then \(\mathbf{K}=0\), \(\overline{\mathbf{S}}=0\), and \(\mathbf{rev}=\lambda\) along \(\nabla\rho\), i.e.\ along the geodesics starting from~\(O\). Moreover, if \(u=v(\rho)\) with \(v'<0\), then one has \(\lambda F^*(Du)=\max\{F^*(\pm Du)\}\).

To remedy the discontinuity at the \(O\), we construct a family \(\mathcal{F}_{0,0,\lambda}=\{(\mathbb{R}^n,F_\varepsilon,\mathsf{m}_\varepsilon)\}_{\varepsilon>0}\) of FMMMs, where
\[
        F_\varepsilon(x,y)=|y|-\frac{\langle x,y\rangle(|x|+2\varepsilon)\theta}{(|x|+\varepsilon)^2},
\]
and \(\mathsf{m}_\varepsilon\) is the Busemann--Hausdorff measure induced by \(F_\varepsilon\). Note that \(F_\varepsilon\to F\) and \(\mathsf{m}_\varepsilon\to \mathsf{m}\) (in densities)  as \(\varepsilon\to 0\). Moreover, for every \(\varepsilon>0\), one has \(\mathbf{K}_\varepsilon\le0\), \(\mathbf{\overline S}_\varepsilon\le0\), and \(\mathbf{rev}_\varepsilon\le\lambda\)  along \(\nabla\rho_\varepsilon\) (see Theorem~\ref{thm:0}).

Using this family, our first result can be stated as follows; for a more general version, including weights and a general order \(p>1\), we refer to  Theorem~\ref{thm:0:app:sharp}.
\begin{theorem}\label{thm:intro:Hardy:sharp}
    Let \(n\ge 3\), \(\lambda\in (1,\infty)\). For every \((\mathbb{R}^n,F_\varepsilon,\mathsf{m}_\varepsilon)\in \mathcal{F}_{0,0,\lambda}\) and \(u\in C_0^\infty(\mathbb{R}^n)\) one has
    \begin{equation}\label{eq:intro:Hardy:sharp}
        \lambda^2\displaystyle\int_{\mathbb{R}^n} F_\varepsilon^*(Du)^2\diff \mathsf{m}_\varepsilon\ge c_n\int_{\mathbb{R}^n}\frac{u^2}{\rho_\varepsilon^2}\diff \mathsf{m}_\varepsilon,\quad\text{where}\quad c_n=\left(\frac{n-2}{2}\right)^2.
    \end{equation}
    Moreover, \(c_n\) is \emph{sharp} in \(\mathcal{F}_{0,0,\lambda}\), in the sense that it is the greatest constant with this property.
\end{theorem}
We highlight that \(\mathbf{K}_\varepsilon\le 0\) is \emph{not} a global property on \((\mathbb{R}^n,F_\varepsilon,\mathsf{m}_\varepsilon)\in \mathcal{F}_{0,0,\lambda}\), i.e., for each point \(x\ne O\), there exists a direction \(y\) (different from \(\nabla\rho(x)\)), such that \(\mathbf{K}_\varepsilon(x,y)>0\). Fortunately, this phenomena does not affect our proofs. In fact, inequality~\eqref{eq:intro:Hardy:sharp} holds for arbitrary manifolds with \(\lambda_F=\lambda\), and \(\mathbf{K}\le 0\), \(\overline{\mathbf{S}}\le0\) along \(\nabla\rho\) (see Theorem~\ref{thm:Hardy:genproof}), thus, the well-known global assumptions may be relaxed.

In the sequel, assuming \emph{strong negative curvature} (\(\mathbf{K}\le-\kappa^2\), \(\kappa>0\) along \(\nabla\rho\)), we provide a similar construction as above. In this case, we focus on the celebrated spectral gap estimate of McKean~\cite{mckean1970upper}: if a Riemannian manifold has strong negative sectional curvature, then there exist a positive, domain independent lower bound for the first eigenvalue of the Laplacian (which is the fundamental tone of the fixed membrane). The Finslerian version of this result is due to Yin and He (see~\cite[Theorem 3.5]{yin2014first}), which can be stated as follows (for \(p=2\)): if \(\mathbf{K}\le -\kappa^2\) and \(\sup_{TM} |\overline{\mathbf{S}}|<(n-1)\kappa\) for some \(\kappa>0\), and \(\lambda_F(M)<\infty\), then
\begin{equation}\label{eq:intro:yinhe}
    \inf_{u\in W^{1,2}_0(M)}\frac{\int_M {F^*}(Du)^2\diff \mathsf{m}}{\int_M|u|^2\diff \mathsf{m}}\ge \left(\frac{(n-1)\kappa-\sup_{TM} |\overline{\mathbf{S}}|}{2\lambda}\right)^2.
\end{equation}

To prove a sharpness result concerning inequality~\eqref{eq:intro:yinhe}, by perturbing the \emph{Klein metric} on the Euclidean unit ball~\(\mathbb{B}^n\), for arbitrary \(\kappa\in(0,\infty)\), \(h\in \mathbb{R}\), and \(\lambda\in(0,\infty)\), we construct a family \(\mathcal{F}_{\kappa,h,\lambda}=\{(\mathbb{B}^n,F_\varepsilon,\mathsf{m}_\varepsilon)\}_{\varepsilon>0}\) satisfying \(\mathbf{K}_\varepsilon\le-\kappa^2\), \(\overline{\mathbf{S}}_\varepsilon\le (n-1)h\), and \(\mathbf{rev}_\varepsilon\le \lambda\) along \(\nabla\rho_\varepsilon\) (see Theorem~\ref{thm:k}). This construction leads to our second main result, which can be stated as follows; for a more general version concerning the first eigenvalue of the \(p\)-Laplacian (with \(p>1\)), we refer to Theorem~\ref{thm:k:app:sharp}.
\begin{theorem}\label{thm:intro:McKean:sharp}
    Let \(n\ge 2\), \(\lambda\in (1,\infty)\), and \(\kappa>h\ge0\). For every \((\mathbb{B}^n,F_\varepsilon,\mathsf{m}_\varepsilon)\in \mathcal{F}_{\kappa,h,\lambda}\) and \(u\in C_0^\infty(\mathbb{B}^n)\) one has
    \begin{equation}\label{eq:intro:McKean:sharp}
        \lambda^2\displaystyle\int_{\mathbb{B}^n} F_\varepsilon^*(Du)^2\diff \mathsf{m}_\varepsilon\ge c_{n,\kappa,h}\int_{\mathbb{B}^n} u^2 \diff \mathsf{m}_\varepsilon,\quad\text{where}\quad c_{n,\kappa,h}=\left(\frac{(n-1)(\kappa-h)}{2}\right)^2.
    \end{equation}
    Moreover, \(c_{n,\kappa,h}\) is \emph{sharp} in \(\mathcal{F}_{\kappa,h,\lambda}\), in the sense that it is the greatest constant with this property.
\end{theorem}
Similarly to the previous case, \(\mathbf{K}\le-\kappa^2\) is not a global property. In addition, inequality~\eqref{eq:intro:McKean:sharp} holds for arbitrary manifolds with~\(\lambda_F=\lambda\), and \(\mathbf{K}\le -\kappa^2\), \(\overline{\mathbf{S}}\le (n-1)h\) along \(\nabla\rho\) (see Theorem~\ref{thm:McKean:genproof}). 

In order to establish inequalities~\eqref{eq:intro:Hardy:sharp}~\&~\eqref{eq:intro:McKean:sharp} on general manifolds, we provide a Finslerian extension of the approach of \emph{Riccati pairs}, developed by Kajántó, Kristály, Peter, and Zhao~\cite{riccatipair2023} in the Riemannian setting, which enables us to prove a general Hardy-type inequality:
% \begin{equation}\label{eq:intro:RP}
\[
   \lambda^p \int_\Omega w(\rho)F(Du)^p\diff \mathsf{m}\ge\int_\Omega w(\rho)\max\{F^*(\pm Du)\}^p\diff \mathsf{m}\ge \int_\Omega w(\rho)W(\rho)|u|^p\diff \mathsf{m},\quad \forall u\in C_0^\infty(\Omega),
\]
%\end{equation}
by simply solving a corresponding Riccati-type ODI, involving the parameters \(w\), \(W\) and \(p\) (see Theorem~\ref{thm:RP}). We note that the proof of this result inspired the basic idea of the construction of the families \(\mathcal{F}_{0,0,\lambda}\) and \(\mathcal{F}_{\kappa,h,\lambda}\). Having this extension, we provide sufficient condition for the sharpness of general Finslerian Hardy-type inequalities (obtained by Riccati pairs) on families \(\mathcal{F}_{0,0,\lambda}\) and \(\mathcal{F}_{\kappa,h,\lambda}\) (see Theorems~\ref{thm:0:sharp}~\&~\ref{thm:k:sharp}). Conceptually, the proof of the sharpness of a Hardy-type inequality relies on a suitable approximation (on the corresponding family) of the \emph{limit function} of the given inequality, which is strongly related to the solution of the Riccati ODI (see Remark~\ref{rem:RP}/\ref{rem:RP:limfunc}).

The paper is structured as follows: In Section~\ref{sec:prelim} we present relevant definitions and results concerning Finsler geometry. In Section~\ref{sec:RP} we provide a Finslerian extension of Riccati pairs by proving Theorem~\ref{thm:RP}, which implies Theorems~\ref{thm:Hardy:genproof}~\&~\ref{thm:McKean:genproof}.
In Section~\ref{sec:0} we address the issue of sharpness in case of non-positive flag curvature: Theorem~\ref{thm:0} presents the family \(\mathcal{F}_{0,0,\lambda}\), while Theorem~\ref{thm:0:sharp} provides sufficient conditions for the sharpness on them. As applications, we prove Theorems~\ref{thm:intro:Hardy:sharp}~\&~\ref{thm:0:app:sharp}.
In Section~\ref{sec:k} we proceed similarly in case of the strong negative flag curvature: Theorem~\ref{thm:k} presents the family \(\mathcal{F}_{\kappa,h,\lambda}\), while Theorem~\ref{thm:k:sharp} provides sufficient conditions for the sharpness on them. As applications, we prove Theorems~\ref{thm:intro:McKean:sharp}~\&~\ref{thm:k:app:sharp}.

% \clearpage
\section{Preliminaries: elements from Finsler geometry}\label{sec:prelim}
In this section, we recall various definitions and results from Finsler geometry that are relevant for our presentations. Here, we also fix our notation.  For further details, we refer to the monographs by Bao, Chern, and Shen~\cite{bao2000introduction} and Shen~\cite{shen2001lectures}.

\subsection{Finsler manifolds}
Let \(M\) be a smooth manifold with dimension \(n\ge 2\). At each point \(x\in M\), the \emph{tangent space} \(T_xM\) consists of tangent vectors (or simply vectors) at \(x\), while the \emph{cotangent space} \(T^*_xM\) consists of cotangent vectors (or covectors) at \(x\), sometimes also referred to as one forms. These spaces are \emph{dual} to each other, the \emph{canonical pairing} between them is a bilinear form defined by
\[
    \langle\xi,y\rangle=\xi(y),\qquad \forall x\in M,\ y\in T_xM,\ \xi\in T^*_xM.
\]

The \emph{tangent bundle} \(TM\) is the collection of all tangent spaces, consisting of \emph{bound vectors}, the \emph{cotangent bundle} \(T^*M\) is the collection of all cotangent spaces, consisting of \emph{bound covectors}. More precisely, we have
\[
    TM=\{(x,y):x\in M, y\in T_xM\}
    \quad\mbox{and}\quad
    T^*M=\{(x,\xi): x\in M, \xi\in T^*_xM\}.
\]

A \emph{vector field} on \(M\) is a map that assigns to each point \(x\in M\) a bound vector \((x,y)\in TM\). Similarly, a \emph{covector field} assigns to each point \(x\in M\) a bound covector \((x,\xi )\in T^*M\).

\begin{remark}\label{rem:prelim:convention}
    For simplicity, we avoid the introduction of additional notation/letters for bound vectors and vector fields. If a statement involves a bound vector \((x,y)\in TM\) or a vector field \(x\mapsto (x,y)\), but the positional coordinate \(x\) is not particularly important, we simply omit \(x\) and formulate the statement using the bound vector \(y\in TM\) or the vector field \(y\), respectively. Similarly, \(\xi\) may denote a covector, a bound covector, or a covector field. The specific underlying objects are always revealed by the context.
\end{remark}

The couple \((M,F)\) is said to be a \emph{Finsler manifold} if the function \(F\colon TM\to[0,\infty)\) is a \emph{Finsler metric} satisfying the following conditions:
\begin{enumerate}[label=\textup{(\roman*)}]
    \item \label{prelim:fdef:i} \emph{Regularity}: \(F(y)\) is smooth on \(TM\setminus\{0\}\), which consists of non-zero bound vectors;
    \item \label{prelim:fdef:ii} \emph{Positive homogeneity}: \(F(\lambda y)=\lambda F(y)\), for every \(y\in TM\) and \(\lambda\ge0\);
    \item \label{prelim:fdef:iii} \emph{Strong convexity}: For every \(y\in TM\setminus\{0\}\), the following quadratic form is positive definite:
          \begin{equation}\label{eq:prelim:g}
              g_{y}(v,w)\defeq\left.\frac{\partial^2}{\partial s\partial t}\left(\frac{1}{2}F(y+sv+tw)^2\right)\right|_{s=t=0},\qquad\forall v,w\in TM.
          \end{equation}
\end{enumerate}

\noindent According to Remark~\ref{rem:prelim:convention}, we use both the notation \(F(y)\) and \(F(x,y)\) depending on the context.

The \emph{reversibility} of a Finsler manifold \((M,F)\),  introduced by Rademacher~\cite{rademacher2004sphere}, is defined by
\[
    \lambda_F(M)\defeq\sup_{x\in M}\sup_{y\in T_xM} \mathbf{rev}(x,y), \quad\mbox{where}\quad \mathbf{rev}(x,y)=\frac{F(x,-y)}{F(x,y)}.
\]
Clearly \(1\le\lambda_F(M)\le +\infty\). The manifold is said to be \emph{reversible}, if \(\lambda_F(M)=1\).

The \emph{dual Finsler metric} of \(F\) is the function \(F^*\colon T^*M\to[0,\infty)\), defined by
\begin{equation}\label{eq:prelim:Fdual}
    F^*(\xi)\defeq \sup_{y\in TM\setminus \{0\}}\frac{\langle\xi,y\rangle}{F(y)},\qquad \forall \xi\in T^*M,
\end{equation}
which is also a Finsler metric on \(M\). The reversibility of \((M,F^*)\) can be defined similarly by
\[
    \lambda_{F^*}(M)\defeq \sup_{x\in M}\sup_{\xi\in T_x^*M} \mathbf{rev}^*(x,y), \quad\mbox{where}\quad \mathbf{rev}^*(x,y)=\frac{F^*(x,-\xi)}{F^*(x,\xi)}.
\]
In particular, according to e.g., Huang, Kristály, and Zhao~\cite[Lemma~2.1]{huang2020sharp}, one has
\begin{equation}\label{eq:lF:lFs}
    \lambda_F(M)=\lambda_{F^*}(M).
\end{equation}

Using a Finsler metric \(F\), one can define the \emph{length} of a Lipschitz continuous curve \(c\colon [a,b]\to M\) by
\[L(c)\defeq\int_a^b F(\dot c(t)) \diff t,\]
where the dot denotes the derivative with respect to \(t\).
Given two points \(x_0,x_1\in M\), the \emph{distance} from \(x_0\) to \(x_1\) is defined by
% \[\]
\(d(x_0,x_1)= \inf_c L(c),\)
where the infimum is taken over all Lipschitz continuous curves \(c\colon [0,1]\to M\) with \(c(0)=x_0\) and \(c(1)=x_1\).
In general, \(d(x_0,x_1)\ne d(x_1,x_0)\), unless \(F\) is reversible.

If \(\rho=d(x_0,\cdot)\) is the distance from a point \(x_0\in M\), then the \emph{eikonal equations} hold \(\!\diff\mathsf{m}\)-a.e.\ in \(M\):
\begin{equation}\label{eq:eikonal}
    F(\nabla \rho)=F^*(D\rho)=1.
\end{equation}
%A function \(\rho\colon \Omega\to (0,\infty)\) is a \emph{distance function} on \(\Omega\) if \(\rho\in W^{1,p}_\textup{loc}(M)\) and relation~\eqref{eq:eikonal} hold \(\!\diff\mathsf{m}\)-a.e.\ in \(M\).

For fixed \(x\in M\), the \emph{Legendre transform} \(\ell^*\colon T^*_xM\to T_xM\) assigns to each covector \(\xi\) the unique vector~\(y\), which minimizes the map
\[
    E^*(y)\defeq \langle \xi, y\rangle - \frac{1}{2}F(y)^2.
\]
The minimizer \(y=\ell^*(\xi)\) can also be interpreted as the unique bound vector  with the properties
\begin{equation}\label{eq:prelim:lt}
    F(\ell^*(\xi))=F^*(\xi)\quad\mbox{and}\quad \langle \xi, \ell^*(\xi)\rangle =F(\ell^*(\xi))F^*(\xi).
\end{equation}
Combining the properties from relation~\eqref{eq:prelim:lt} yields
\begin{equation}\label{eq:prelim:lt:cor}
    \langle \xi,\ell^*(\xi)\rangle=F^*(\xi)^2, \qquad\forall \xi\in T^*_xM.
\end{equation}

Similarly, the Legendre transform \(\ell\colon T_xM\to T^*_xM\) assigns to each vector \(y\) the unique covector \(\xi\) that minimizes the map
\[
    E(\xi)\defeq \langle \xi, y\rangle - \frac{1}{2}F^*(\xi)^2.
\]
For the minimizer \(\xi=\ell(y)\) similar relations hold as in~\eqref{eq:prelim:lt}~\&~\eqref{eq:prelim:lt:cor}. One also has \(\ell^{-1}=\ell^*\), hence \(y=\ell^*(\xi)\). Using the fact that \(\frac{\partial E}{\partial \xi}=0\), we obtain
\begin{equation}\label{eq:prelim:grF}
    \ell^*(\xi)=F^*(\xi)\frac{\partial  F^*(\xi)}{\partial \xi},\qquad
    \forall \xi\in T_x^*M.
\end{equation}
Thus, \(\ell^*\) is positive homogeneous. Note that the definitions and properties of the Legendre transforms \(\ell\) and \(\ell^*\) naturally extend to bound vectors and vector fields. Moreover, using the conventions of Remark~\ref{rem:prelim:convention}, all of the above formulas are formally valid. We only need to make sure that the positional coordinates of \(y\) and \(\xi\) coincide, otherwise the canonical pairing does not make sense.

For every \(x\in M\), there exists a local coordinate system \(\{x^i\}\) defined on a coordinate neighborhood of \(x\). Denote by \(\{\frac{\partial}{\partial x^i}\}\) and \(\{\!\diff x^i\}\) the induced basis on \(T_xM\) and \(T^*_xM\), respectively.

A \emph{Finsler metric measure manifold} \((M,F, \mathsf{m})\) (abbreviated by FMMM) is a Finsler manifold equipped with a smooth positive measure \(\mathsf{m}\). Unlike the Riemannian setting, there is no canonical measure in Finsler geometry. However, one often considers the Busemann--Hausdorff measure induced by \(F\), which is given by
\begin{equation}\label{eq:prelim:BH}
    \diff \mathsf{m}_F\defeq\sigma_F(x)\diff x^1\wedge\ldots\wedge \diff x^n,\quad\mbox{where}\quad  \sigma_F(x)=\frac{\operatorname{vol}(\mathbb{B}^n)}{\operatorname{vol}(B_x(1))},
\end{equation}
\(\operatorname{vol}(\cdot)\) denotes the Euclidean volume, \(\mathbb{B}^n\subset \mathbb{R}^n\) is the Euclidean unit ball, while \[
    B_x(1)=\{y\in T_xM: F(y)<1\}.
\]
The \emph{derivative} of a smooth function \(u\colon M\to \mathbb{R}\) is a covector field defined by
\begin{equation}\label{eq:prelim:Du}
    D u\defeq \frac{\partial u}{\partial x^i}\diff x^i.
\end{equation}
The \emph{gradient} of a smooth function \(u\colon M\to \mathbb{R}\) is a vector field that can be computed as
\begin{equation}\label{eq:prelim:gru}
    \nabla u\defeq\ell^*(D u).
\end{equation}
% \[
% \]
The \emph{divergence} of a smooth vector field \(v\) with components \((v^i)\) is given by
\[
    \di(v)\defeq \sum_i \frac{1}{\sigma}\frac{\partial \sigma v^i}{\partial x^i}.
\]
If \(p>1\), then the \emph{\(p\)-Laplacian} of a smooth function \(u\) is
\[
    \Delta_{F,p} (u)\defeq \di(F(\nabla u)^{p-2}\nabla u).
\]
We note that for \(p=2\), the \(p\)-Laplacian is precisely the Finsler Laplacian
\[\Delta_F (u) = \di(\nabla u).\]
We also notice that, if \(\rho\) is a distance function, then
\(\Delta_{F,p}(\rho)=\Delta_F(\rho)\) holds \(\!\diff\mathsf{m}\)-a.e.\ in \(M\).

Finally, if \(u_1\) and \(u_2\) are smooth functions, both vanishing at the boundary, then the following \emph{integration by parts formula} holds:
\begin{equation}\label{eq:prelim:ibp}
    \int_M u_1\Delta_F (u_2)\diff\mathsf{m}=-\int_M \langle D u_1,\nabla u_2\rangle\diff\mathsf{m},
\end{equation}
see e.g., Ohta and Sturm~\cite{ohta2009heat}.

\subsection{Curvatures and Laplace comparison} In the sequel, we recall the notions of the flag and \(S\)-curvature of a FMMM \((M,F,\mathsf{m})\), which can be interpreted more easily using tensors. According to the notation of the previous section, the fundamental tensor from relation~\eqref{eq:prelim:g} can be rewritten as
\[
    g_{(x,y)}=g_{ij}(x,y)\diff x^i\diff x^j,\quad\mbox{where}\quad g_{ij}(x,y)\defeq\frac{1}{2}\frac{\partial^2 F(x,y)^2}{\partial y^i\partial y^j},\qquad \forall (x,y)\in TM.
\]
Note that, by convention, repeated indexes are automatically summed. The \emph{geodesic spray coefficients} are defined as follows:
\[
    G^i(x,y)\defeq\frac{1}{4}g^{il}(x,y)\left(2\frac{\partial g_{jl}(x,y)}{\partial x^k}- \frac{\partial g_{jk}(x,y)}{\partial x^l}\right)y^jy^k,\qquad \forall (x,y)\in TM.
\]
Recall that a smooth curve \(t\mapsto \gamma(t)\) is \emph{geodesic} if it satisfies the equation
\[
    \ddot \gamma^i(t)+2G^i(\gamma(t),\dot \gamma(t))=0.
\]

The \emph{Riemannian curvature} is defined by
\[
    R_{(x,y)}\defeq R^i_k(x,y)\frac{\partial}{\partial x^i} \otimes dx^k, \qquad \forall (x,y)\in TM,
\]
where
\[
    R^i_k(x,y)\defeq  2 \frac{\partial G^i(x,y)}{\partial x^k}-y^i\frac{\partial^2 G^i(x,y)}{\partial x^j\partial y^k}+2G^j(x,y)\frac{\partial^2 G^i(x,y)}{\partial y^j\partial y^k}-\frac{\partial  G^i(x,y)}{\partial y^j}\frac{\partial  G^j(x,y)}{\partial y^k}.
\]
The definition of \emph{flag curvature} is as follows:
\begin{equation}\label{eq:prelim:K}
    \mathbf{K}(x,y,v)=\frac{g_{(x,y)}(R_{(x,y)}(v),v)}{g_{(x,y)}(y,y)g_{(x,y)}(v,v)-g_{(x,y)}(y,v)^2},\qquad \forall x\in M,\ y,v\in T_xM.
\end{equation}
The name ``flag'' refers to the plane spanned by the vectors \(y\) and \(v\); the vector \(y\) is sometimes referred to as the \emph{flagpole}.

The \emph{distortion} is defined by
\[
    \tau(x,y)\defeq\log\frac{\sqrt{\det g_{ij}(x,y)}}{\sigma(x)},\qquad \forall (x,y)\in TM.
\]
The \emph{\(S\)-curvature} measures the rate of change of the distortion along geodesics, more precisely
\begin{equation}\label{eq:prelim:S}
    \mathbf{S}(x,y)=\left.\frac{\diff}{\diff t}\tau(\gamma(t),\dot\gamma(t))\right|_{t=0},\qquad \forall (x,y)\in TM,
\end{equation}
where \(\gamma\) denotes the unique geodesic with \(\gamma(0)=x\) and \(\dot\gamma(0)=y\). For computational convenience (due to the homogeneity of degree zero), we prefer to use the \emph{reduced \(S\)-curvature} defined by
\[
    \overline{\mathbf{S}}(x,y)=\frac{\mathbf{S}(x,y)}{F(x,y)},\quad \forall (x,y)\in TM,\  y\ne 0.    
\]
%It is well known that geodesics are locally minimizer of the distance functions
For every \(\kappa\ge 0\), define \(\mathbf{ct}_\kappa\colon (0,\infty)\to \mathbb{R}\) by
\[{\mathbf{ct}}_\kappa(t)=
    \begin{cases}
        \frac{1}{t},           & \mbox{if } \kappa=0, \\
        \kappa\coth(\kappa t), & \mbox{if } \kappa>0, \\
    \end{cases}.
\]

% q is homogenous, rho is radial
% \begin{equation}\label{eq:prelim:rho}
%     q(\rho(x))\defeq q(x,\nabla \rho)=q(x,x)
% \end{equation}

Our forthcoming arguments rely deeply on the following Laplace comparison theorem, which is a slight modification of the celebrated result by Wu and Xin~\cite[Theorem 5.1]{wu2007comparison}. Here we only assume curvature conditions alongside the geodesics starting from some \(x_0\in M\), which does not affect the original proof.
\begin{theorem}\label{thm:comp}
    Let \((M,F)\) be a Finsler manifold and \(\rho=d(x_0,\cdot)\) be the distance from a point \(x_0\in M\). Suppose that \(\mathbf{K}\le -\kappa^2\) and \(\overline{\mathbf{S}}\le (n-1)h\) along \(\nabla\rho\), for some \(\kappa\ge 0\) and \(h\in \mathbb{R}\). Then one has
    \[
        \Delta_F (\rho)\ge (n-1)\mathbf{ct}_\kappa(\rho)-(n-1)h.
    \]
    Moreover, equality holds if and only if \(\mathbf{K}= -\kappa^2\) and \(\overline{\mathbf{S}}= (n-1)h\), along \(\nabla\rho\).
\end{theorem}

\subsection{Special Finsler metrics} \label{sec:prelim:randers}

A Finsler metric on \(\Omega\subseteq\mathbb{R}^n\) is said to be \emph{projectively flat} if the geodesics are straight (Euclidean) lines as point sets. Due to Rapcsák~\cite{rapcsak}, we have the following characterization.
\begin{theorem}\label{thm:flat}
    Let \(F\colon T\Omega\to[0,\infty)\) be a Finlser metric on \(\Omega\subseteq \mathbb{R}^n\). The following conditions are equivalent:
    \begin{enumerate}[label=\textup{(\roman*)},labelindent=0cm,wide]
        \item \label{thm:flat:i} The Finsler space \((\Omega,F)\) is projectively flat.
        \item \label{thm:flat:ii} The geodesic coefficients of \(F\) are given by
              \begin{equation}\label{eq:flat:factor}
                  G^i(x,y)=P(x,y)y^i,\quad\mbox{where}\quad P(x,y)=\frac{1}{2F(x,y)}\cdot \frac{\partial F(x,y)}{\partial x^k}y^k.
              \end{equation}
        \item \label{thm:flat:iii} The metric \(F\) satisfies
              \begin{equation}\label{eq:flat:condition}
                  \frac{\partial^2 F(x,y)}{\partial x^k\partial y^i}y^k=\frac{\partial F(x,y)}{\partial x^i},\quad \forall i\in\{1,\dots,n\}.
              \end{equation}
    \end{enumerate}
\end{theorem}

The following proposition is useful for computing the flag and \(S\)-curvature of a projectively flat metric. For the proof, we refer to Shen~\cite[Sections 6.3 and 7.3]{shen2001lectures}
\begin{proposition}\label{prop:flat:computations} Suppose that \((M,F,\mathsf{m})\) is a FMMM equipped with a projectively flat Finsler metric \(F\) and a smooth positive measure \(\mathsf{m}\) having density \(\sigma\). The flag and \(S\)-curvature can be computed as follows:
    \begin{align*}
        \mathbf{K}(x,y) & =\frac{1}{F(x,y)^2}\left(P(x,y)^2-\frac{\partial P(x,y)}{\partial x^i}y^i\right),                     \\
        \mathbf{S}(x,y) & =\frac{\partial G^i(x,y)}{\partial y^i}-\frac{y^i}{\sigma(x)}\frac{\partial \sigma(x)}{\partial x^i},
    \end{align*}
    where \(P\) and \(G^i\) are given by relation~\eqref{eq:flat:factor}. Note that in this case \(\mathbf{K}\) only depends on \(x\) and \(y\).
\end{proposition}

In the sequel, we consider a special class of Finsler spaces, called \emph{Randers spaces}, which are important examples of Finsler manifolds: In many cases, they help us to observe and investigate conceptual differences between Riemannian and Finsler geometries. A Randers space is a differentiable manifold \(M\) endowed with a special Finsler metric \(F\colon TM\to[0,\infty)\) called \emph{Randers metric}, which is defined by
\begin{equation}\label{eq:randers:def}
    F(x,y)=\sqrt{a_{ij}(x)y^iy^j}+b_i(x)y^i,
\end{equation}
where \(a_{ij}(x)\diff x^i\diff x^j\) is a Riemannian metric, \(b_i(x)\diff x^i\) is a a \emph{one-form}, such that
% \begin{equation}\label{eq:randers:condition}
\[
    \|b(x)\|_a\defeq\sqrt{a^{ij}(x)b_i(x)b_j(x)}<1, \quad\forall x\in M,
\]
% \end{equation}
and \(\{a^{ij}(x)\}\) denotes the inverse of \(\{a_{ij}(x)\}\).
The density of the Busemann--Hausdorff measure induced by a Raders metric \(F\) is given by
\begin{equation}\label{eq:Randers:measure}
    \sigma_F(x)=(1-\|b(x)\|_a^2)^\frac{n+1}{2}\sqrt{\det a_{ij}(x)}.
\end{equation}
For later use, define the \emph{Kronecker delta} by
\begin{equation}\label{eq:prelim:kronecker}
    \delta_{ij}\defeq \begin{cases}
        1, & \mbox{if } i=j,    \\
        0, & \mbox{if } i\ne j.
    \end{cases}
\end{equation}

The following observation proves to be useful in the sequel.
\begin{remark}\label{rem:Randers:flat} Let \(F\) be a Randers metric given  by~\eqref{eq:randers:def}. It the Riemannian metric \(\overline F(x,y)=\sqrt{a_{ij}(x)y^iy^j}\) satisfies~\eqref{eq:flat:condition}, and \(b_i\diff x^i=D\beta(x)\) for some smooth function \(\beta\), then \(F\) satisfies~\eqref{eq:flat:condition} as well. Thus, according to Theorem~\ref{thm:flat} the metric \(F\) is projectively flat. The claim easily follows from the symmetry of second order partial derivatives of the smooth function \(\beta\).
\end{remark}

% \clearpage
\section{A Finslerian extension of Riccati pairs}\label{sec:RP}
In this section, we present a Finslerian extension of the approach of Riccati pairs, established by Kajántó, Kristály, Peter, and Zhao~\cite{riccatipair2023} in the Riemannian setting, which enables us to prove a Hardy-type inequality, by simply solving a corresponding Riccati-type ordinary differential inequality. The key ideas of the proof are similar to~\cite{riccatipair2023}, however, some extra workarounds are needed due to the specificities of the Finsler geometry. Our extension can be stated as follows.

\begin{theorem}\label{thm:RP}
    Let \((M,F,\mathsf{m})\) be a forward complete, non-compact FMMM, with dimension \(n\ge 2\) and reversibility \(\lambda\in(1,\infty)\). Let \(\Omega\subseteq M\) be a domain, \(p>1\), \(x_0\in M\) and \(\rho=d(x_0,\cdot)\) the distance from \(x_0\).
    Suppose that the couple \((L,W)\) is a \emph{\((p,\rho,w)\)-Riccati pair} in \((0,\sup_\Omega\rho)\) that is
    \begin{enumerate}[label=\textup{({\bfseries C\arabic*})},leftmargin=*]
        \item\label{cond:RP:c1}  \(L,W,w\colon(0,\sup_\Omega\rho)\to(0,\infty)\) such that \(L,W\) are continuous and \(w\) is of class \(C^1\);
        \item\label{cond:RP:c2}  \(\Delta_F \rho\geq L(\rho)\) in the distributional sense in \(\Omega\);
        \item\label{cond:RP:c3} there exists a function \(G\colon(0,\sup_\Omega\rho)\to(0,\infty)\) of class \(C^1\) such that the following Riccati ODI holds:
              \begin{equation}\label{eq:RP:ODI}
                  (G(t)w(t))'+G(t)w(t)L(t)-(p-1)G(t)^{p'}w(t)\geq  W(t)w(t), \quad\forall t\in (0,\sup\nolimits_\Omega\rho),
              \end{equation}
    \end{enumerate}
    where \(p'=\frac{p}{p-1}\). Then for every \(u\in C_0^\infty(\Omega)\) one has
    \begin{align}\label{eq:RP}
        \lambda^p\int_\Omega w(\rho)F^*(Du)^p\diff \mathsf{m}\ge \int_\Omega w(\rho)\max\{F^*(\pm Du)\}^p\diff \mathsf{m} & \ge \int_\Omega w(\rho)W(\rho)|u|^p\diff \mathsf{m}.
    \end{align}
\end{theorem}
\begin{proof}
    The convexity of the map \(\xi\mapsto F^*(\xi)^p\) implies
    \begin{equation}\label{eq:RP:convexity:orig}
        % \[
        F^*(\xi)^p\ge  F^*(\eta)^p +\left\langle \xi-\eta, \frac{\partial}{\partial \eta}F^*(\eta)^p\right\rangle,\qquad \forall \xi,\eta\in T^*\Omega.
        % \]
    \end{equation}
    By using the chain rule and the properties of Legendre transforms (see relations~\eqref{eq:prelim:lt:cor}~\&~\eqref{eq:prelim:grF}), we obtain
    \begin{equation}\label{eq:RP:convexity}
        F^*(\xi)^p\ge p F^*(\eta)^{p-2}\langle \xi,\ell^*(\eta)\rangle-(p-1)F^*(\eta)^p,\qquad \forall \xi,\eta\in T^*\Omega.
    \end{equation}
    Introduce the notation \[
        \operatorname{sgn}(t) = \begin{cases}
            1  & \text{if }t>0,  \\
            0  & \text{if } t=0, \\
            -1 & \text{if } t<0,
        \end{cases}
    \]
    and let us make in inequality~\eqref{eq:RP:convexity} the following choices:
    \begin{equation}\label{eq:RP:choices}
        \xi=-\operatorname{sgn}(u)Du\quad\mbox{and}\quad\eta=\operatorname{sgn}(u)G(\rho)^\frac{1}{p-1}u D \rho.
    \end{equation}
    On the one hand, the positive homogeneity of \(F^*\) implies
    \begin{align*}
        p F^*(\eta)^{p-2} & =pG(\rho)^\frac{p-2}{p-1}F^*(D\rho)^{p-2}|u|^{p-2}, \\
        (p-1)F^*(\eta)^p  & =(p-1)G(\rho)^{p'}F^*(D\rho)^p|u|^{p}.
    \end{align*}
    On the other hand, since \(\ell^*\) is positive homogeneous and the canonical pairing is bilinear, we have
    \[
        \langle \xi,\ell^*(\eta)\rangle=-G(\rho)^\frac{1}{p-1}u\langle Du, \ell^*(D\rho)\rangle.
    \]
    Thus, the chain rule yields
    \[
        pF^*(\eta)^{p-2}\langle \xi,\ell^*(\eta)\rangle=-G(\rho)F^*(D\rho)^{p-2}\langle D(|u|^p), \ell^*(D\rho)\rangle.
    \]
    By the above computations, it follows that
    % \begin{equation}%\label{eq:main:beforeInt:1}
    \[
        F^*(-\operatorname{sgn}(u)Du)^p\ge -G(\rho)F^*(D\rho)^{p-2}\langle D(|u|^p), \ell^*(D\rho)\rangle-(p-1)G(\rho)^{p'}F^*(D\rho)^p|u|^{p}.
    \]
    Multiplying both sides by \(w(\rho)>0\) and integrating over \(\Omega\) yields
    \begin{align}
        \int_\Omega w(\rho)F^*(-\operatorname{sgn}(u)Du)^p\diff \mathsf{m}
         & \ge -\int_\Omega w(\rho)G(\rho)F^*(D\rho)^{p-2}\langle D (|u|^p), \ell^*(D\rho)\rangle\diff \mathsf{m}\nonumber \\
         & \qquad -(p-1)\int_\Omega w(\rho)G(\rho)^{p'}F^*(D\rho)^p|u|^{p}\diff \mathsf{m}.\label{eq:main:afterInt}
        %,\nonumber\\
        %&\defeq pI_H+(1-p)J_H.%\label{eq:main:afterInt}
    \end{align}
    First, by the eikonal equation~\eqref{eq:eikonal}, the factors \(F^*(D\rho)^{p-2}\) and \(F^*(D\rho)^p\) can be omitted. Second, the integration by parts formula~\eqref{eq:prelim:ibp} for \(u_1=w(\rho)G(\rho)|u|^p\) and \(u_2=\rho\) implies
    \[
        \int_\Omega (w(\rho)G(\rho)\Delta_F\rho) |u|^p\diff \mathsf{m}=-\int_\Omega(w(\rho)G(\rho))'\langle D\rho,\nabla\rho\rangle |u|^p\diff \mathsf{m}-\int_\Omega w(\rho)G(\rho)\langle D(|u|^p),\nabla\rho\rangle\diff \mathsf{m}.
    \]
    Next, since \(\nabla\rho=\ell^*(D\rho)\), by relation~\eqref{eq:prelim:lt} and the eikonal equation, the factor \(\langle D\rho,\nabla\rho\rangle\) can also be omitted. Thus, inequality~\eqref{eq:main:afterInt} can be rewritten as
    \[
        \int_\Omega w(\rho) F^*(-\operatorname{sgn}(u)Du)^p\diff \mathsf{m}\ge \int_\Omega \left((w(\rho)G(\rho))'+w(\rho)G(\rho)\Delta_F\rho-(p-1)w(\rho)G(\rho)^{p'}\right)|u|^p\diff \mathsf{m}.
    \]
    Next, since \(w\) and \(G\) are positive, condition~\ref{cond:RP:c2} implies
    \[
        \int_\Omega w(\rho)F^*(-\operatorname{sgn}(u)Du)^p\diff \mathsf{m}\ge \int_\Omega \left((w(\rho)G(\rho))'+w(\rho)G(\rho)L(\rho)-(p-1)w(\rho)G(\rho)^{p'}\right)|u|^p\diff \mathsf{m},
    \]
    thus, applying the Riccati ODI~\eqref{eq:RP:ODI} from condition~\ref{cond:RP:c3} yields
    \[
        \int_\Omega w(\rho)F^*(-\operatorname{sgn}(u)Du)^p\diff \mathsf{m}\ge \int_\Omega w(\rho)W(\rho)\diff \mathsf{m}.
    \]
    Finally, since \(\lambda_F(M)=\lambda_{F^*}(M)=\lambda\) (see relation~\eqref{eq:lF:lFs}), the following chain of inequalities hold:
    \begin{equation}\label{eq:RP:lambdaexplanation}
        \lambda F^*(Du)\ge \max\{F^*(\pm Du)\}\ge F^*(-\operatorname{sgn}(u)Du),
    \end{equation}
    which combined with the positivity of \(w\) finishes the proof.
\end{proof}

Several comments are in order.

\begin{remark}\phantomsection\label{rem:RP}
    \begin{enumerate}[label=\textup(\alph*),wide,labelindent=0cm,itemsep=1mm]
        \item\label{rem:RP:conditions} The above theorem, in fact, extends the combination of Theorems 3.1~\&~3.2 of~\cite{riccatipair2023}. For a more straightforward presentation we imposed more simpler but stricter conditions on the parameters: we assumed smoothness instead of local integrability, eliminated an extra parameter, considered only distances from a point, and so on. However, if the applications require, one can easily relax these conditions, without affecting the proof.

        \item\label{rem:RP:extension}
              Conceptually, Theorem~\ref{thm:RP} has the same benefit as its Riemannian counterpart: It reduces the proof of a Hardy inequality from the relation~\eqref{eq:RP} to the solvability of the Riccati ODI~\eqref{eq:RP:ODI}. In addition, the Riemannian applications presented in~\cite{riccatipair2023} gain natural Finslerian extensions: One only needs to formally replace \(|\nabla u|\) by either \(\max\{F^*(\pm Du)\}\) or \(\lambda F^*(Du)\) in the corresponding inequalities.

        \item \label{rem:RP:uncertainty}
              In~\cite{riccatipair2023}, the authors also provided an alternative formulation of the above result called \emph{multiplicative form}, which allows us to consider uncertainty principles (such as the Heisenberg--Pauli--Weyl uncertainty principle and the Hydrogen uncertainty principle) and Caffarelli--Kohn--Nirenberg inequalities. The proof relies on a scaling argument, which can also be adapted to the Finsler setting. By doing this, one can provide Finslerian extensions to the aforementioned inequalities. The details are left to the interested reader.
        \item 
              Theorem~\ref{thm:RP} highlights the effect of \(\mathbf{K}\), \(\overline{\mathbf{S}}\), and \(\mathbf{rev}\) on Hardy inequalities: Condition~\ref{cond:RP:c2} is ensured by the Laplace comparison theorem (see Theorem~\ref{thm:comp}), the exact form of~\(L\) depends on the upper bound of the flag curvature \(\mathbf{K}\) and the reduced \(S\)-curvature \(\overline{\mathbf{S}}\) along the geodesics starting from \(x_0\). The effect of reversibility is made precise in relation~\eqref{eq:RP:lambdaexplanation}.
        \item \label{rem:RP:limfunc} %\label{rem:RP2} 
              Theorem~\ref{thm:RP} also gives some hints about the sharpness of Hardy-type inequalities: one can expect the inequalities of~\eqref{eq:RP} to be \emph{close} to the equality if the inequality of condition~\ref{cond:RP:c2}, the inequalities from~\eqref{eq:RP:lambdaexplanation}, and the convexity inequality~\eqref{eq:RP:convexity:orig} are close to equality. The first two of them are related to the geometry of the ambient space (see also Theorem~\ref{thm:comp}), while the third provides information concerning \(u\).

              Define the \emph{limit function} \(u^\star\) of the Hardy inequality~\eqref{eq:RP}, as a radially symmetric function \(u^\star=v(\rho)\) for some \(v\colon(0,\sup_\Omega\rho)\to \mathbb{R}\), for which \(\xi=\eta\) in relation~\eqref{eq:RP:choices}, and thus inequality~\eqref{eq:RP:convexity:orig} becomes an equality. Then we obtain
              \[
                  -(\log(v(t)))'=-\frac{{v}'(t)}{v(t)}=G(t)^\frac{1}{p-1},\quad \forall t\in(0,\sup\nolimits_\Omega\rho).
              \]
              For \(u=u^\star\) the integrals of the inequality~\eqref{eq:RP} are not convergent (there exists no extremal function), but considering its approximations/truncations may yield sharpness.
              To get close to the equality in~\ref{cond:RP:c2}~\&~\eqref{eq:RP:lambdaexplanation}, in Sections~\ref{sec:0:model}~\&~\ref{sec:k:model}, we construct some suitable spaces for these approximations. We note, however, that a general proof of sharpness (if exists) would not require such constructions.
    \end{enumerate}
\end{remark}

\subsection{Applications} In this section, by applying Theorem~\ref{thm:RP}, we prove a Finslerian version of a weighted Hardy inequality and a McKean-type spectral gap estimate. We note that our choice of parameters agree with the corresponding Riemannian counterparts from~\cite{riccatipair2023}. The first application can be stated as follows.
\begin{theorem}\label{thm:Hardy:genproof}
    Let \((M,F,\mathsf{m})\) be a forward complete, non-compact FMMM, with dimension \(n\ge 3\), and reversibility \(\lambda<\infty\). Let \(\Omega\subseteq M\) be a domain, \(x_0\in M\) and \(\rho=d(x_0,\cdot)\) the distance from \(x_0\). Suppose that \(\mathbf{K}\le 0\), \(\overline{\mathbf{S}}\le 0\) along \(\nabla \rho\). Let \(\alpha,p\in \mathbb{R}\) such that \(1<p<n+\alpha\).
    Then for every \(u\in C_0^\infty(\Omega)\) one has
    \begin{equation}\label{eq:Hardy:genproof}
        \lambda^p\displaystyle\int_\Omega \rho^\alpha\cdot F^*(Du)^p\diff \mathsf{m}    \ge\displaystyle\int_\Omega \rho^\alpha\cdot \max\{F^*(\pm Du)\}^p\diff \mathsf{m} \ge  c_{n,p,\alpha}\displaystyle\int_\Omega \rho^\alpha\frac{|u|^p}{\rho^p}\diff \mathsf{m},
    \end{equation}
    where \(c_{n,p,\alpha}=\left(\frac{n+\alpha-p}{p}\right)^p\).
\end{theorem}
\begin{proof}
    In Theorem~\ref{thm:RP} let us choose \(\Omega=M\setminus\{x_0\}\), as well as,
    \[w(t)=t^\alpha,\quad L(t)=\frac{n-1}{t},\quad W(t)= c_{n,p,\alpha}\cdot \frac{1}{t^p} ,\quad\mbox{and}\quad G(t)=(c_{n,p,\alpha})^\frac{p-1}{p}\cdot\frac{1}{t^{p-1}},\quad\forall t>0.\]
    Obviously, the above functions satisfy~\ref{cond:RP:c1}. The Laplace comparison (see Theorem~\ref{thm:comp}) implies that~\ref{cond:RP:c2} holds as well. By a simple computation, the Riccati ODI of condition~\ref{cond:RP:c3} is verified with equality. Thus, Theorem~\ref{thm:RP} implies that inequality~\eqref{eq:Hardy:genproof} holds for every \(u\in C_0^\infty(M\setminus\{x_0\})\). Finally, we use a result from capacity theory (see e.g., \cite[Remark 4.2]{riccatipair2023}): Since the \(p\)-capacity of \(\rho^{-1}(0)\subset M\) is zero, thus, inequality~\eqref{eq:Hardy:genproof} holds for every \(u\in C_0^\infty(M)\) as well.
\end{proof}
We notice that Zhao in~\cite[Theorem 1.1]{Z} showed that if \(\mathbf{K}\le 0\) (globally) and \(\mathbf{S}\ge0\) along \(\nabla \overleftarrow\rho\), where \(\overleftarrow\rho=d(\cdot, x_0)\), then one has
\[
    \int_\Omega  \overleftarrow\rho^\alpha\cdot \max\{F^*(\pm Du)\}^2\diff \mathsf{m}\ge c_{n,p,\alpha}\int_\Omega\overleftarrow\rho^\alpha\frac{|u|^p}{\overleftarrow\rho^p}\diff \mathsf{m},\quad \forall  u\in C_0^\infty(M);
\]
additionally, if \(F\) is reversible, then the constant \(c_{n,p,\alpha}\) is sharp. Since, \(\mathbf{S}\ge 0\) along \(\nabla\overleftarrow\rho\) if and only if \(\mathbf{S}\le 0\) along \(\nabla \rho\), this result perfectly aligns with inequality~\eqref{eq:Hardy:genproof}.

The second application can be stated as follows.

\begin{theorem}\label{thm:McKean:genproof}
    Let \((M,F,\mathsf{m})\) be a forward complete, non-compact FMMM, with dimension \(n\ge 3\), and reversibility \(\lambda<\infty\). Let \(\Omega\subseteq M\) be a domain, \(x_0\in M\) and \(\rho=d(x_0,\cdot)\) the distance from \(x_0\). Suppose that \(\mathbf{K}\le -\kappa^2\), \(\overline{\mathbf{S}}\le (n-1)h\) along \(\nabla \rho\), for some \(\kappa>h\ge0\), and let \(p>1\).
    Then for every \(u\in C_0^\infty(\Omega)\) one has
    \begin{equation}\label{eq:McKean:genproof}
        \lambda^p\int_\Omega F^*(Du)^p\diff \mathsf{m}  \ge\int_\Omega \max\{F^*(\pm Du)\}^p\diff \mathsf{m}  \ge  c_{n,p,\kappa,h} \int_\Omega |u|^p\diff \mathsf{m},
    \end{equation}
    where \(c_{n,p,\kappa,h}=\left(\frac{(n-1)(\kappa-h)}{p}\right)^p\).
\end{theorem}
\begin{proof}
    In Theorem~\ref{thm:RP} let us choose \(\Omega=M\), as well as
    \[
        w(t)\equiv 1, \quad L(t)\equiv (n-1)(\kappa-h) , \quad W(t)\equiv c_{n,p,\kappa,h},\quad\mbox{and}\quad G(t)\equiv  (c_{n,p,\kappa,h})^\frac{p-1}{p}.
    \]
    These functions satisfy~\ref{cond:RP:c1}, the Laplace comparison theorem (see Theorem~\ref{thm:comp}) implies
    \[
        \Delta_F (\rho)\ge (n-1)(\kappa\coth(\kappa \rho)-h)\ge (n-1)(\kappa-h) = L,
    \]
    hence~\ref{cond:RP:c2} holds as well. The Riccati ODI from~\ref{cond:RP:c3} is verified with inequality, thus, Theorem~\ref{thm:RP} proves inequality~\eqref{eq:McKean:genproof}.
\end{proof}

\section{Sharp Hardy inequalities -- non-positive curvature}\label{sec:0}
In this section, we present our sharpness results concerning Hardy-type inequalities assuming non-positive flag curvature. In Section~\ref{sec:0:model}, we present a {family} of Finsler spaces on which our sharpness results are established (see Theorem~\ref{thm:0}). In Section~\ref{sec:0:conditions}, we provide some sufficient conditions ensuring the sharpness of Hardy inequalities obtained by Riccati pairs, on the aforementioned spaces (see Theorem~\ref{thm:0:sharp}). Finally, in Section~\ref{sec:0:applications} we present an application: In Theorem~\ref{thm:0:app:sharp} we prove a sharpness result concerning a weighted Hardy inequality, as a consequence we also obtain Theorem~\ref{thm:intro:Hardy:sharp}.

In the rest of the paper, we use various properties of projectively flat Randers-type metrics; we refer to Section~\ref{sec:prelim:randers} for a brief summary on these special manifolds.

\subsection{Construction}\label{sec:0:model} According to the proof of Theorem~\ref{thm:RP} (see also Remark~\ref{rem:RP}/\ref{rem:RP:limfunc}) an \emph{ideal} space for discussing the sharpness of Hardy-type inequalities has a radially symmetric metric \(F\) and constant \(\mathbf{K}\), \(\overline{\mathbf{S}}\), and \(\mathbf{rev}\) along geodesics starting for a point \(x_0\). Unfortunately, such spaces are not available in the Finsler setting, however, we can make the following observation.

Consider the triple \((\mathbb{R}^n,F,\mathsf{m})\), where \(F\colon \mathbb{R}^n\times \mathbb{R}^n\to [0,\infty)\) is defined by
    \begin{equation}\label{eq:0:Fbase}
        F(x,y)=|y|-\frac{\langle x,y\rangle\theta}{|x|},\quad\text{if } x\ne O\quad\text{and}\quad F(O,y)=|y|,
    \end{equation}
    for some \(\theta\in(0,1)\) and \(\mathsf{m}\) stands for the Busemann--Hausdorff measure induced by \(F\). Here, \(O\in \mathbb{R}^n\) is the origin, \(|\cdot|\) and \(\langle\cdot,\cdot \rangle\) denote the norm and the inner product in \(\mathbb{R}^n\). The function \(F\) is {not} a Finsler metric, since it is not continuous at the origin. However, it has various remarkable properties (away from \(O\)).

    On the one hand, since relation~\eqref{eq:flat:condition} holds, Theorem~\ref{thm:flat} yields projectively flatness, i.e.\ all the geodesics are straight (Euclidean) lines. Denote \(\rho=d(O,\cdot)\) the Finslerian distance from the origin. Since the geodesics starting from the origin are precisely the integral curves of \(\nabla \rho\),  if follows that \(x\) is parallel with \(\nabla\rho\). Using Proposition~\ref{prop:flat:computations} and the homogeneity of the flag curvature and the reduced \(S\)-curvature we obtain
    \begin{equation}\label{eq:0:ks}
        \mathbf{K}(x,\nabla\rho(x))=\mathbf{K}(x,x)=0\quad\text{and}\quad \overline{\mathbf{S}}(x,\nabla\rho(x))=\overline{\mathbf{S}}(x,x)=0,\quad \forall x\in \mathbb{R}^n, x\ne O,
    \end{equation}
    which precisely means that \(\mathbf{K}=0\), \(\overline{\mathbf{S}}=0\) along \(\nabla \rho\).

    On the other hand, by an easy computation we obtain for every \(x,y\in  \mathbb{R}^n\) that
    \begin{equation}\label{eq:0:lam}
        \textbf{rev}(x,y)\le \textbf{rev}(x,x)=\frac{1+\theta}{1-\theta}\defeq \lambda,\quad\text{if }x\ne O,\quad\text{and}\quad \mathbf{rev}(O,y)=1,
    \end{equation}
    hence \(\lambda_F(\mathbb{R}^n)=\lambda\in(1,\infty)\). If \(u=v(\rho)\) is a radially symmetric function with \(v'<0\), then one has equality in~\eqref{eq:RP:lambdaexplanation}. The latter provides the difference between the two common forms of Hardy inequalities in the Finsler setting, compare e.g., inequalities~\eqref{eq:intro:Hardy:Finsler}~\&~\eqref{eq:intro:Hardy:Finsler:lam}.

    To remedy the discontinuity at the origin, one could simply puncture the space (considering only \(\mathbb{R}^n\setminus\{O\}\)), which results in loss of computational convenience: By construction, the distance from \(O\) is easier to  compute than the distance from any other base point.
    An other option (what is chosen here) is to `smooth out' the jump at the origin, which can be done by a careful approximation of \(F\), partially retaining the aforementioned favorable properties. We have the following result.

    \begin{theorem}\label{thm:0} Let \(\lambda\in (1,\infty)\) and \(\theta=\frac{\lambda-1}{\lambda+1}\). Construct the family \(\mathcal{F}_{0,0,\lambda}=\{(\mathbb{R}^n,F_\varepsilon,\mathsf{m}_\varepsilon)\}_{\varepsilon>0}\) of FMMMs, where  \(F_\varepsilon\colon T\mathbb{R}^n\simeq\mathbb{R}^n\times\mathbb{R}^n\to [0,\infty)\) is defined by
        \begin{equation}\label{eq:0:Fe}
            F_\varepsilon(x,y)=|y|-\frac{\langle x,y\rangle(|x|+2\varepsilon)\theta}{(|x|+\varepsilon)^2},
        \end{equation}
        and \(\mathsf{m}_\varepsilon\) is the Busemann--Hausdorff measure induced by \(F_\varepsilon\). The following statements hold:
        \begin{enumerate}[label=(\roman*)]
            \item For every \(\varepsilon>0\), the triple \((\mathbb{R}^n, F_\varepsilon,\mathsf{m}_\varepsilon)\) is a Finsler metric measure manifold. In addition, \(F_\varepsilon\) is a projectively flat Randers metric, with reversibility \(\lambda_{F_\varepsilon}(\mathbb{R}^n)=\lambda\).
            \item For every \(x\in \mathbb{R}^n\) with \(x\ne 0\) one has
                  \[
                      \mathbf{K}_\varepsilon(x,x)\to 0,\quad \overline{\mathbf{S}}_\varepsilon(x,x)\to 0,\quad\mbox{and}\quad \mathbf{rev}_\varepsilon(x,x)\to\lambda,\quad\text{as}\quad  \varepsilon\to 0.
                  \]
            \item For every \(\varepsilon>0\) and \(x\in \mathbb{R}^n\), one has
                  \[
                      \mathbf{K}_\varepsilon(x,x)\le 0,\quad \overline{\mathbf{S}}_\varepsilon(x,x)\le 0,\quad\mbox{and}\quad
                      \mathbf{rev}_\varepsilon(x,x)\le\lambda.
                  \]
        \end{enumerate}
    \end{theorem}%\clearpage
    \begin{proof}
        Denote \(r=|x|\). Using the standard notation of Randers metrics (see Section~\ref{sec:prelim:randers}) we have
        \[
            F_\varepsilon(x,y)=\sqrt{a_{ij}(x)y^iy^j}+b_i(x)y^i,\quad\mbox{where}\quad a_{ij}(x)=\delta_{ij}\quad\mbox{and}\quad b_i(x)=-\frac{x^i(r+2\varepsilon)\theta}{(r+\varepsilon)^2},
        \]
        while \(\delta_{ij}\) is the Kronecker delta (see~\eqref{eq:prelim:kronecker}).  Since \(a^{ij}(x)=\{a_{ij}(x)\}^{-1}=\delta_{ij}\), we have
        \[\|b(x)\|_a = \sqrt{a^{ij}(x)b_i(x)b_j(x)}=\frac{r(r+2\varepsilon)\theta}{(r+\varepsilon)^2}<\theta<1.\]
        Obviously \(F_\varepsilon\) is smooth, hence it is a Randers metric on \(\mathbb{R}^n\).
        In addition, according to Remark~\ref{rem:Randers:flat}, the metric \(F_\varepsilon\) is projectively flat. Indeed, the Euclidean metric is projectively flat, and one has
        \[
            b_i\diff x^i=D\beta(x),\quad\mbox{where}\quad \beta(x)=-\left(r+\frac{\varepsilon^2}{r+\varepsilon}\right)\theta.
        \]
        Let \(O\in \mathbb{R}^n\) be the origin. For \(\rho_\varepsilon(x)=d(O,x)\) and \(\overleftarrow{\rho_\varepsilon} (x)=d(x,O)\) we obtain:
        \[
            \rho_\varepsilon(x)=r\cdot \frac{r(1-\theta)+ \varepsilon }{r+\varepsilon}\quad\mbox{and}\quad
            \overleftarrow{\rho_\varepsilon} (x) =r\cdot \frac{r(1+\theta)+ \varepsilon }{r+\varepsilon}.
        \]
        In addition, one has
        \begin{equation}\label{eq:0:drho}
             D\rho_\varepsilon(x)=\left(\frac{1}{r}-\frac{  (r+2 \varepsilon )\theta}{(r+\varepsilon )^2}\right)\cdot x\quad\mbox{and}\quad\nabla\rho_\varepsilon(x)=\frac{(r+\varepsilon)^2}{r(r+\varepsilon)^2-r^2(r+2\varepsilon)\theta}\cdot x.
        \end{equation}
        According to relation~\eqref{eq:Randers:measure}, the density of the Busemann--Hausdorff measure \(\mathsf{m}_\varepsilon\) induced by \(F_\varepsilon\) is
        \[
            \sigma_{F_\varepsilon}(x)=\left(1-\frac{r^2 (r+2 \varepsilon)^2\theta^2}{(r+\varepsilon)^4}\right)^\frac{n+1}{2}.
        \]
        By Proposition~\ref{prop:flat:computations}, for every \(\varepsilon>0\) and \(x\in \mathbb{R}^n\) we obtain
        \begin{align*}
            \mathbf{K}_{\rho_\varepsilon}(r)
             & \defeq\mathbf{K}_\varepsilon(x,x)=-\frac{3 \varepsilon ^2 (r+\varepsilon)^4   \theta (1-\theta)}{\left((r+\varepsilon)^2- r (r+2 \varepsilon)\theta \right)^4}\le 0,                                                        \\
            \overline{\mathbf{S}}_{\rho_\varepsilon}(r)
             & \defeq
            \overline{\mathbf{S}}_\varepsilon(x,x)=-\frac{(n+1)\varepsilon ^2(r+\varepsilon ) \theta }{\left(r+\varepsilon \right)^4- r^2 \left(r+2 \varepsilon\right)^2\theta^2}\le 0,                                                    \\
            \mathbf{rev}_{\rho_\varepsilon}(r)
             & \defeq\mathbf{rev}_\varepsilon(x,x)=\frac{1+\|b(x)\|_a}{1-\|b(x)\|_a}=\frac{1+\frac{r(r+2\varepsilon)\theta}{(r+\varepsilon)^2}}{1-\frac{r(r+2\varepsilon)\theta}{(r+\varepsilon)^2}}\le \frac{1+\theta}{1-\theta}=\lambda.
        \end{align*}
        Considering the limit cases of the above inequalities, we obtain
        \[
            \lim_{r\to\infty}\mathbf{K}_{\rho_\varepsilon}(r)=0,\quad
            \lim_{r\to\infty}\overline{\mathbf{S}}_{\rho_\varepsilon}(r)=0,\quad
            \lim_{r\to\infty}\mathbf{rev}_{\rho_\varepsilon}(r)=\lambda,\quad \forall \varepsilon>0,
        \]
        and
        \[
            \lim_{\varepsilon\to0}\mathbf{K}_{\rho_\varepsilon}(r)=0,\quad
            \lim_{\varepsilon\to0}\overline{\mathbf{S}}_{\rho_\varepsilon}(r)=0,\quad \lim_{\varepsilon\to0}\mathbf{rev}_{\rho_\varepsilon}(r)=\lambda,\quad \forall r>0.
        \]
        Moreover, for every \(\varepsilon>0\) and \(x,y\in \mathbb{R}^n\) one has
        \begin{equation}\label{eq:0:revprop}
            \mathbf{rev}_\varepsilon(x,y)=\frac{F_\varepsilon(x,-y)}{F_\varepsilon(x,y)}\le \frac{F_\varepsilon(x,-x)}{F_\varepsilon(x,x)}=\mathbf{rev}_\varepsilon(x,x)\le \lambda,\quad\text{if }x\ne O\quad\text{and}\quad
            \mathbf{rev}_\varepsilon(O,y)=1,
        \end{equation}
        thus, on also has \(\lambda_{F_\varepsilon}(\mathbb{R}^n)=\lambda\), for every \(\varepsilon>0\).
    \end{proof}

    We make the following of observations.
    \begin{remark}\phantomsection\label{rem:0}
        \begin{enumerate}[label=\textup(\alph*),wide,labelindent=0cm,itemsep=1mm]
            \item The above construction can be motivated as follows. First, Randers-type metrics are often considered as a first step towards Finsler geometry from Riemannian geometry, it is quite natural to consider them in this case as well.  Second, every radially symmetric Randers metric with respect to some \(x_0\), needs to have \(\mathbf{rev}(x_0,\cdot)=1\) (as in~\eqref{eq:0:revprop}), otherwise it would be discontinuous at \(x_0\). Since our \(F\) in~\eqref{eq:0:Fbase} has constant \(\mathbf{rev}\ne 1\) along \(\nabla \rho\), some smoothing/perturbation is unavoidable at \(x_0\).  Indeed, one can consider other perturbations instead of the one in~\eqref{eq:0:Fe}. The latter is chosen due to its simplicity and computational convenience.

            \item\label{rem:0:b} One can consider a measure \(\mathsf{m}_\varepsilon\) with density \(\sigma_\varepsilon=\sigma_F\cdot e^{-s\rho_\varepsilon}\), to obtain
                  \[
                      \sup_{x\in \mathbb{R}^n}\overline{\mathbf{S}}_\varepsilon(x,x)\le s\quad\mbox{and}\quad \lim_{\varepsilon\to 0} \overline{\mathbf{S}}_\varepsilon(x,x)=s,\quad \forall x\in \mathbb{R}^n\setminus\{x_0\}.
                  \]
                  In this section we only need \(s=0\). This observation, however, is exploited in Section~\ref{sec:k}.

            \item\label{rem:0:c} We note that in case of the reduced \(S\)-curvature, \(\overline{\mathbf{S}}\le0\) holds also globally. Surprisingly, in case of the flag curvature, \(\mathbf{K}\le 0\) is not a global property. For example, one has
                  \[
                      \mathbf{K}_\varepsilon(x,-x)=\frac{3 \varepsilon ^2 (r+\varepsilon)^4   \theta (1+\theta)}{\left((r+\varepsilon)^2+ r (r+2 \varepsilon)\theta \right)^4}>0,\quad \forall x\in \mathbb{R}^n.
                  \]
            \item\label{rem:0:d} Observe that \(\max\{F^{*}_\varepsilon(\pm D\rho_\varepsilon)\}=F^{*}_\varepsilon(D\rho_\varepsilon)=1\). Indeed, relation~\eqref{eq:eikonal} implies \(F^*_\varepsilon(D\rho_\varepsilon)=1\), and one has 
                  \[
                      F^*_\varepsilon(-D\rho_\varepsilon)=\frac{1-\frac{r(r+2\varepsilon)}{(r+\varepsilon)^2}\theta}{1+\frac{r(r+2\varepsilon)}{(r+\varepsilon)^2}\theta}\defeq h_\varepsilon(r),\quad\mbox{where}\quad \frac{1-\theta}{1+\theta}\le h_\varepsilon(r)\le1.
                  \]
                Note that \(F^*_\varepsilon\) can be computed as in Shen~\cite[Example 3.1.1]{shen2001lectures}, while \(D\rho_\varepsilon\) is given by relation~\eqref{eq:0:drho}.
        \end{enumerate}
    \end{remark}

    % \clearpage
    \subsection{Sufficient conditions for sharpness}\label{sec:0:conditions}
    In this section, we present sufficient conditions for the sharpness of Hardy inequalities provided by Riccati pairs on the spaces of \(\mathcal{F}_{0,0,\lambda}\), introduced in Theorem~\ref{thm:0}.

    Let \((\mathbb{R}^n,F_\varepsilon,\mathsf{m}_\varepsilon)\in\mathcal{F}_{0,0,\lambda}\) for some \(\varepsilon>0\) and \(\lambda\in(1,\infty)\). Then for every  \(x\in \mathbb{R}^n\), one has
    \[
        \mathbf{K}_\varepsilon(x,x)\le 0,\quad \overline{\mathbf{S}}_\varepsilon(x,x)\le 0,\quad\mbox{and}\quad
        \mathbf{rev}_\varepsilon(x,x)\le\lambda=\lambda_{F_\varepsilon}(\mathbb{R}^n).
    \]
    Denote \(\rho=d(O,\cdot)\) the distance from the origin. The Laplace comparison theorem (see Theorem~\ref{thm:comp}) implies
    \[
        \Delta_{F_\varepsilon}(\rho_\varepsilon)\ge \frac{n-1}{\rho_\varepsilon}\defeq L^{0,0}(\rho_\varepsilon).
    \]
    Let \(R\in(0,+\infty]\), \(L(t)\le L^{0,0}(t)\) for every \(t\in(0,R)\), and define
% \(\Omega_\varepsilon=\{x\in \mathbb{R}^n\mid \rho_\varepsilon(x)<R\}.\)
\[\Omega_\varepsilon=\{x\in \mathbb{R}^n : \rho_\varepsilon(x)<R\}.\]

Suppose that
\((L,W)\) is a \emph{\((p,\rho_\varepsilon,w)\)-Riccati pair} in \((0,R)\), that is \(W\colon(0,R)\to (0,\infty)\) is continuous, \(w\colon(0,R)\to (0,\infty)\) is of class \(C^1\), and there exists \(G\colon(0,R)\to(0,\infty)\) of class \(C^1\) such that
\begin{equation}\label{eq:0:RP:odi}
    (G(t)w(t))'+G(t)w(t)L(t)-(p-1)G(t)^{p'}w(t)\geq  W(t)w(t), \quad\forall t\in (0,R).
\end{equation}
Then by Theorem~\ref{thm:RP}, for every \(u\in C_0^\infty(\Omega_\varepsilon)\) one has:
\begin{align}\label{eq:0:RP:int:lam}
    \lambda^p\int_{\Omega_\varepsilon} w(\rho_\varepsilon) F^*_\varepsilon(Du)^p\diff \mathsf{m}_\varepsilon    & \ge\int_{\Omega_\varepsilon} w(\rho_\varepsilon) W(\rho_\varepsilon) |u|^p\diff \mathsf{m}_\varepsilon.
\end{align}
As in Remark~\ref{rem:RP}/\ref{rem:RP:limfunc}, suppose that a function \(G>0\) verifies the Riccati ODI~\eqref{eq:0:RP:odi} with equality, and define \(v\colon (0,R)\to\mathbb{R}\) by
\begin{equation}\label{eq:vdef:0}
    -(\log v(t))'=-\frac{v'(t)}{v(t)}=G(t)^\frac{1}{p-1}.
\end{equation}
Since \(G\) is smooth and positive, it follows that \(v\) is smooth, positive, and decreasing. Introduce the notation
\begin{equation}\label{eq:0:I}
    \mathcal{I}(f;a,b)=\int_a^b f(t)t^{n-1}\diff t,
\end{equation}
for arbitrary continuous functions \(f\colon[a,b]\to\mathbb{R}\), with \(0<a<b<R\).

We have the following result.
\begin{theorem}\label{thm:0:sharp}
    Fix \(\delta_0>0\) and let \(t_1=t_1(\delta),t_2=t_2(\delta),t_3=t_3(\delta),t_4=t_4(\delta)\) be continuous functions admitting a (possibly infinite) limit at infinity, such that for every \(\delta>\delta_0\) one has \(0<t_1<t_2<t_3<t_4<R\). Suppose that the following limits hold true:
    \begin{align*}
        \lim_{\delta\to\infty} l_0(\delta) & =0,\quad\mbox{where}\quad l_0(\delta)\defeq \frac{v(t_2)^p\cdot \mathcal{I}(w;t_1,t_2)}{(t_2-t_1)^p\cdot \mathcal{I}(wWv^p;t_2,t_3)}+\frac{v(t_3)^p\cdot \mathcal{I}(w;t_3,t_4)}{(t_4-t_3)^p\cdot \mathcal{I}(wWv^p;t_2,t_3)}, \\
        \lim_{\delta\to\infty} l_1(\delta) & =1,\quad\mbox{where}\quad l_1(\delta)\defeq\max_{t\in [t_2,t_3]} \frac{G(t)^{p'}}{W(t)}.
    \end{align*}
    Then inequality~\eqref{eq:0:RP:int:lam} is sharp on \(\mathcal{F}_{0,0,\lambda}\), in the sense that it no longer holds on each FMMM of \(\mathcal{F}_{0,0,\lambda}\) if we multiply its right hand side by a constant \(c>1\).
\end{theorem}
\begin{proof}
    Let \(\varepsilon>0\) and define the \emph{limit function} \(u^\star_\varepsilon\colon \Omega\to (0,\infty)\) of inequality~\eqref{eq:0:RP:int:lam} by \[u^\star_\varepsilon(x)=v(\rho_\varepsilon(x)).\]
    Denote \(T=(t_1,t_2,t_3,t_4)\) and construct
    \[u^\star_{\varepsilon,T}=v_T(\rho_\varepsilon),\]
    where the function \(v_T\) and its derivative satisfy
    \begin{equation}\label{eq:0:trunc}
        v_T(t)=\begin{cases}
            0,                           & \mbox{if }0<t\le t_1,      \\
            v(t_2)\frac{t-t_1}{t_2-t_1}, & \mbox{if }t_1\le t\le t_2, \\
            v(t),                        & \mbox{if }t_2\le t\le t_3, \\
            v(t_3)\frac{t_4-t}{t_4-t_3}, & \mbox{if }t_3\le t\le t_4, \\
            0,                           & \mbox{if }t_4\le t<R,
        \end{cases}\quad\mbox{and}\quad v_T'(t)=\begin{cases}
            0,                       & \mbox{if }0<t<t_1,   \\
            \frac{v(t_2)}{t_2-t_1},  & \mbox{if }t_1<t<t_2, \\
            v'(t),                   & \mbox{if }t_2<t<t_3, \\
            -\frac{v(t_3)}{t_4-t_3}, & \mbox{if }t_3<t<t_4, \\
            0,                       & \mbox{if }t_4< t<R.
        \end{cases}
    \end{equation}
    Note that \(v'_T\) is not defined in \(t_1,t_2,t_3\) and \(t_4\). In addition, both \(v_T\) and \(v'_T\) are supported in \([t_1,t_4]\). For a schematic illustration of \(v_T\) we refer to Figure~\ref{fig:0:trunc}.
    \begin{figure}
        \includegraphics[width=0.5\textwidth]{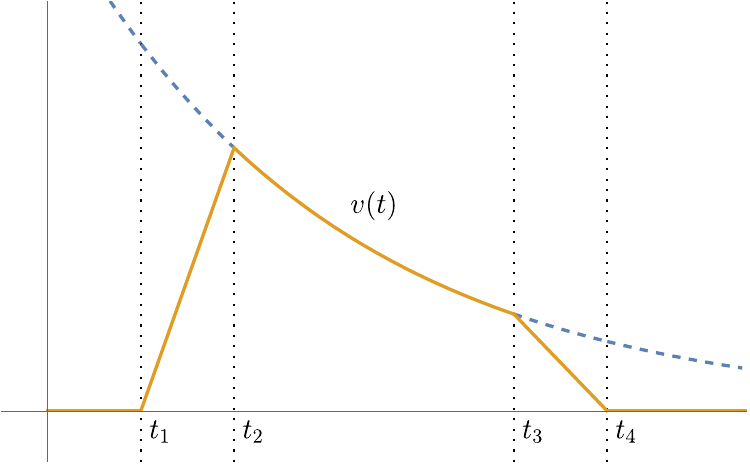}
        \caption{Illustration of the ``tent-like'' function \(v_T(t)\): coincides with \(v(t)\) on \([t_2,t_3]\),\\ linearized on both \([t_1,t_2]\) and \([t_3,t_4]\), and zero elsewhere.}
        \label{fig:0:trunc}
    \end{figure}

    To prove the sharpness of inequality~\eqref{eq:0:RP:int:lam}, we show that under a suitable limiting procedure, one has 
    \[
        \frac{\lambda^p\displaystyle\int_{\Omega_\varepsilon} w(\rho_\varepsilon) F^*_\varepsilon(Du^\star_{\varepsilon,T})^p\diff \mathsf{m}_\varepsilon}{
        \displaystyle\int_{\Omega_\varepsilon} w(\rho_\varepsilon) W(\rho_\varepsilon) (u_{\varepsilon,T}^\star)^p\diff \mathsf{m}_\varepsilon}\to 1.
    \]
    As a first step, we use polar transforms to reduce the above integrals to one dimension. We also establish some fine estimates on the distance and the density of the measure.

    Using the notations \(r=|x|\) and \(t=\rho_\varepsilon(x)\), it follows that
    \[t=r\cdot \frac{r(1-\theta)+ \varepsilon }{r+\varepsilon}\quad\iff\quad r=\frac{t-\varepsilon+\sqrt{(t+\varepsilon)^2-4t\varepsilon\theta}}{2(1-\theta)}\defeq \varphi_\varepsilon(t),\]
    hence we obtain
    \begin{equation}\label{eq:0:dm}
        \sigma_\varepsilon(r)r^{n-1}\diff r=\sigma_\varepsilon(\varphi_\varepsilon(t))\varphi_\varepsilon(t)^{n-1}\varphi_\varepsilon'(t)\diff t\defeq\psi_\varepsilon(t)\diff t.
    \end{equation}
    Since for every \(t\in(0,R)\) one has
    \[
        t\le \varphi_\varepsilon(t)\le\frac{t}{1-\theta},\quad 1\le \varphi'_\varepsilon(t)\le \frac{1}{1-\theta},\quad\mbox{and}\quad\left(1-\theta^2\right)^\frac{n+1}{2}\le \sigma_\varepsilon(\varphi_\varepsilon(t))\le 1,
    \]
    we obtain
    \begin{equation}\label{eq:0:cdef}
        ct^{n-1}\defeq\left(1-\theta^2\right)^\frac{n+1}{2}t^{n-1}\le \psi_\varepsilon(t)\le \frac{t^{n-1}}{(1-\theta)^n}\defeq Ct^{n-1}.
    \end{equation}
    Note that \(c\) and \(C\) are independent of \(\varepsilon\).

    Since \(v\) is positive and decreasing, the positive homogeneity of \(F^*_\varepsilon\) and the relations from~\eqref{eq:0:trunc} imply
    \begin{equation}\label{eq:0:F:trunc}
        F^*_\varepsilon(Du^\star_{\varepsilon,T})=F^*_\varepsilon(v_T'(\rho_\varepsilon)D\rho_\varepsilon)=\begin{cases}
            \frac{v(t_2)}{t_2-t_1}F^*_\varepsilon(D\rho_\varepsilon), & \text{if } \rho_\varepsilon\in(t_1,t_2),         \\
            |v'(\rho_\varepsilon)|F^*_\varepsilon(-D\rho_\varepsilon) & \text{if } \rho_\varepsilon\in(t_2,t_3),         \\
            \frac{v(t_3)}{t_4-t_3}F^*_\varepsilon(-D\rho_\varepsilon) & \text{if } \rho_\varepsilon\in(t_3,t_4),         \\
            0,                                                        & \text{if }\rho_\varepsilon\in(0,t_1)\cup(t_4,R).
        \end{cases}
    \end{equation}
    According to Remark~\ref{rem:0}/\ref{rem:0:d}, one has
    \[F^*_\varepsilon(D\rho_\varepsilon)=1,\quad\mbox{and}\quad F^*_\varepsilon(-D\rho_\varepsilon)=\frac{1-\frac{r(r+2\varepsilon)}{(r+\varepsilon)^2}\theta}{1+\frac{r(r+2\varepsilon)}{(r+\varepsilon)^2}\theta}\defeq h_\varepsilon(r),\quad\mbox{where}\quad \frac{1-\theta}{1+\theta}\le h_\varepsilon(r)\le1.\]
    By a simple computation, we obtain that \(h_\varepsilon(r)\) is decreasing in \(r\), increasing in \(\varepsilon\), and approaches to \(\frac{1-\theta}{1+\theta}=\frac{1}{\lambda}\) as \(\varepsilon\to 0\). Since \(r=\varphi_\varepsilon(t)\ge t\), we also have
    \begin{equation}\label{eq:0:hred}
        h_\varepsilon(\varphi_\varepsilon(t))\le h_\varepsilon(t)\le 1.
    \end{equation}

    In the sequel, denote by \(\omega\) the volume of the Euclidean unit ball of dimension \(n\).

    On the one hand, relations~\eqref{eq:0:dm}~and~\eqref{eq:0:F:trunc} yield
    \begin{align*}
        I\defeq\frac{1}{n\omega} \int_{\Omega_\varepsilon} w(\rho_\varepsilon) F^*_\varepsilon(Du^\star_{\varepsilon,T})^p\diff \mathsf{m}_\varepsilon &=\frac{v(t_2)^p}{(t_2-t_1)^p}\int_{t_1}^{t_2} w(t) \psi_\varepsilon(t)\diff t + \int_{t_2}^{t_3} w(t)|v'(t)|^p h_\varepsilon^p(\varphi_\varepsilon(t)) \psi_\varepsilon(t)\diff t
   \\
        &\qquad  +\frac{v(t_3)^p}{(t_4-t_3)^p}\int_{t_3}^{t_4} w(t) h_\varepsilon^p(\varphi_\varepsilon(t))\psi_\varepsilon(t)\diff t.
    \end{align*}
    Applying relations~\eqref{eq:0:cdef}~\&~\eqref{eq:0:hred}, and using the definition of \(\mathcal{I}\) from~\eqref{eq:0:I}, it follows that
    \begin{align*}
        I & \le \frac{v(t_2)^p}{(t_2-t_1)^p}\int_{t_1}^{t_2} w(t) \psi_\varepsilon(t)\diff t+ \int_{t_2}^{t_3} w(t)|v'(t)|^p h_\varepsilon^p(t) \psi_\varepsilon(t)\diff t +\frac{v(t_3)^p}{(t_4-t_3)^p}\int_{t_3}^{t_4} w(t) h_\varepsilon^p(t)\psi_\varepsilon(t)\diff t
        \\&\le \frac{Cv(t_2)^p}{(t_2-t_1)^p}\int_{t_1}^{t_2} w(t) t^{n-1}\diff t+ h_\varepsilon^p(t_2)\int_{t_2}^{t_3} w(t)|v'(t)|^p \psi_\varepsilon(t)\diff t +\frac{C h_\varepsilon^p(t_3)v(t_3)^p}{(t_4-t_3)^p}\int_{t_3}^{t_4} w(t)t^{n-1}\diff t\\
                                         & =\frac{Cv(t_2)^p\cdot \mathcal{I}(w;t_1,t_2)}{(t_2-t_1)^p}+ h_\varepsilon^p(t_2)\int_{t_2}^{t_3} w(t)|v'(t)|^p \psi_\varepsilon(t)\diff t +\frac{C v(t_3)^p\cdot\mathcal{I}(w;t_3,t_4)}{(t_4-t_3)^p}.
    \end{align*}
    By relation~\eqref{eq:vdef:0}, for every \(t\in[t_2,t_3]\) one has
    \[
        \frac{1}{W(t)}\left|\frac{v'(t)}{v(t)}\right|^p=\frac{1}{W(t)}\left(-\frac{v'(t)}{v(t)}\right)^p=\frac{G(t)^{p'}}{W(t)}.
    \]
    The above function is continuous, thus, it attains its maximum on \([t_2,t_3]\). Hence, we have
    \begin{align*}
         \int_{t_2}^{t_3} w(t)|v'(t)|^p \psi_\varepsilon(t)\diff t&=\int_{t_2}^{t_3} \frac{1}{W(t)}\left|\frac{v'(t)}{v(t)}\right|^p w(t) W(t)|v(t)|^p \psi_\varepsilon(t)\diff t\\
         &\le  \max_{t\in[t_2,t_3]} \frac{G(t)^{p'}}{W(t)}\cdot \int_{t_2}^{t_3} w(t) W(t) (v(t))^p \psi_\varepsilon(t)\diff t.
    \end{align*}

    On the other hand, we obtain
    \begin{align}
        \nonumber J&\defeq 
        \frac{1}{n\omega}
        \displaystyle\int_{\Omega_\varepsilon} w(\rho_\varepsilon) W(\rho_\varepsilon) (u_{\varepsilon,T}^\star)^p\diff \mathsf{m}_\varepsilon=\int_{0}^{R} w(t) W(t) (v_T(t))^p \psi_\varepsilon(t)\diff t\\\label{eq:0:estimates} &\,\ge \int_{t_2}^{t_3} w(t) W(t) (v(t))^p \psi_\varepsilon(t)\diff t\ge c\cdot \mathcal{I}(wWv^p;t_2,t_3).
    \end{align}
    
    Combining the above results, and using both estimates of~\eqref{eq:0:estimates}, it follows that
    \begin{align*}
        \frac{\lambda^p I}{J}
         & \le \lambda^p h_\varepsilon^p(t_2)\cdot  \max_{t\in[t_2,t_3]} \frac{G(t)^{p'}}{W(t)}+\frac{\lambda^p C}{c}\left(\frac{v(t_2)^p\cdot \mathcal{I}(w;t_1,t_2)}{(t_2-t_1)^p\cdot \mathcal{I}(wWv^p;t_2,t_3)}+\frac{v(t_3)^p\cdot \mathcal{I}(w;t_3,t_4)}{(t_4-t_3)^p\cdot \mathcal{I}(wWv^p;t_2,t_3)}\right) \\
         & =\lambda^p h_\varepsilon^p(t_2) l_1(\delta)+\frac{\lambda^p C}{c}l_0(\delta).
    \end{align*}
    Observe that the above inequality holds for every \(\varepsilon>0\) and \(\delta>\delta_0\), in addition, \(\delta\), \(T=(t_1,t_2,t_3,t_3)\), \(c\), \(C\), \(\lambda\), \(l_0\), and \(l_1\) are independent of \(\varepsilon\). To prove the sharpness of~\eqref{eq:0:RP:int:lam}, it is enough to choose for every \(\delta>\delta_0\) an \(\varepsilon=\varepsilon(\delta)\) such that \(h_{\varepsilon}(t_2(\delta))\to\frac{1}{\lambda}\) as \(\delta\to\infty\). Indeed, in this case, since \(l_0\to0\) and \(l_1\to1\) as \(\delta\to\infty\) we obtain \(\lambda^pI/J\to1\) as \(\delta\to \infty\).

    By assumption there exists \[\lim_{\delta\to\infty} t_2(\delta)\defeq l\in[0,R].\] We distinguish two cases: If \(l>0\), then choose \(\varepsilon=1/\delta\), to obtain
    \[
        \lim_{\delta\to \infty} h_{\varepsilon}(t_2(\delta))= \lim_{\delta\to \infty}\frac{1-\frac{t_2(\delta)(t_2(\delta)+2/\delta)}{(t_2(\delta)+1/\delta)^2}\theta}{1+\frac{t_2(\delta)(t_2(\delta)+2/\delta)}{(t_2(\delta)+1/\delta)^2}\theta}=\frac{1-\theta}{1+\theta}=\frac{1}{\lambda}.
    \]
    If \(l=0\), then choosing \(\varepsilon=t_2(\delta)^2\) yields
    \[
        \lim_{\delta\to \infty} h_{\varepsilon}(t_2(\delta))= \lim_{t_2\to 0}\frac{1-\frac{t_2(t_2+2t_2^2)}{(t_2+t_2^2)^2}\theta}{1+\frac{t_2(t_2+2t_2^2)}{(t_2+t_2^2)^2}\theta}=\lim_{t_2\to 0}\frac{1-\frac{(1+2t_2)}{(1+t_2)^2}\theta}{1+\frac{(1+2t_2)}{(1+t_2)^2}\theta}=\frac{1-\theta}{1+\theta}=\frac{1}{\lambda},
    \]
    which concludes the proof.
\end{proof}

\subsection{The sharpness of a weighted Hardy inequality}\label{sec:0:applications} In this section, we present a sharpness result concerning a weighted Hardy inequality on \(\mathcal{F}_{0,0,\lambda}\). As a consequence we also obtain Theorem~\ref{thm:intro:Hardy:sharp}. Our result can be stated as follows.
\begin{theorem}\label{thm:0:app:sharp}
    Let \(n\ge 3\), \(\lambda\in (1,\infty)\), and \(\alpha,p\in \mathbb{R}\) such that \(1<p<n+\alpha\). For every \((\mathbb{R}^n,F_\varepsilon,\mathsf{m}_\varepsilon)\in \mathcal{F}_{0,0,\lambda}\) and \(u\in C_0^\infty(\mathbb{R}^n)\) one has
    \begin{equation}\label{eq:0:app:Hardy}
        \lambda^p\displaystyle\int_{\mathbb{R}^n} \rho^\alpha \cdot F_\varepsilon^*(Du)^p\diff \mathsf{m}_\varepsilon\ge c_{n,p,\alpha}\int_{\mathbb{R}^n}\rho^\alpha\frac{u^p}{\rho_\varepsilon^p}\diff \mathsf{m}_\varepsilon,\quad\text{where}\quad c_{n,p,\alpha}=\left(\frac{n+\alpha-p}{p}\right)^p.
    \end{equation}
    Moreover, \(c_{n,p,\alpha}\) is \emph{sharp} in \(\mathcal{F}_{0,0,\lambda}\), in the sense that it is the greatest constant with this property.
\end{theorem}
\begin{proof}
    The proof of inequality~\eqref{eq:0:app:Hardy} follows from Theorem~\ref{thm:Hardy:genproof}. To prove the sharpness, we use Theorem~\ref{thm:0:sharp}. According to the relation~\eqref{eq:vdef:0}, we obtain
    \[v(t)^p=\frac{1}{t^{n+\alpha-p}},\quad\forall t\in(0,\infty).\]
    For every \(\delta>1\) define
    \[t_1(\delta)=\delta/2,\quad t_2(\delta)=\delta,\quad t_3(\delta)=\delta^2,\quad\mbox{and}\quad t_4(\delta)=2\delta^2.\]
    By simple computations we obtain
    \begin{align*}
        \frac{v(t_2)^p\cdot \mathcal{I}(w;t_1,t_2)}{(t_2-t_1)^p} & =\frac{2^{n+\alpha}-1}{(n+\alpha)2^{n+\alpha-p}}=\text{const},        \\
        \frac{v(t_3)^p\cdot \mathcal{I}(w;t_3,t_4)}{(t_4-t_3)^p} & =\frac{2^{n+\alpha}-1}{n+\alpha}=\text{const},                        \\
        \mathcal{I}(wWv^p;t_2,t_3)                               & =c_{n,p,\alpha}\log(\delta)\to\infty,\quad\text{as } \delta\to\infty.
    \end{align*}
    It follows that \(l_0\to 0\) as \(\delta\to\infty\). Since \(l_1=\frac{G(t)^{p'}}{W(t)}=1\), for every \(t>0\), thus, Theorem~\ref{thm:0:sharp} yields the sharpness of the inequality~\eqref{eq:0:app:Hardy}.
\end{proof}
We observe that choosing \(\alpha=0\) and \(p=2\) in the above result yields Theorem~\ref{thm:intro:Hardy:sharp}.

\section{Sharp Hardy inequalities -- strong negative curvature}\label{sec:k}

In this section, we present our sharpness results concerning Hardy-type inequalities, assuming strong negative curvature. We proceed similarly to the previous case. In Section~\ref{sec:k:model} we present a second construction of spaces (see Theorem~\ref{thm:k}). In Section~\ref{sec:k:conditions} we provide sufficient conditions for sharpness of Hardy-type inequalities on them (see Theorem~\ref{thm:k:sharp}). Finally, in Section~\ref{sec:k:applications} we present an application: In Theorem~\ref{thm:k:app:sharp} we prove a sharpness result concerning a McKean-type special gap estimate, which implies Theorem~\ref{thm:intro:McKean:sharp}.

\subsection{Construction}\label{sec:k:model} Consider the triple \((\mathbb{B}^n,F, \mathsf{m})\), where \(\mathbb{B}^n\) denotes the Euclidean (open) unit ball, the function \(F\colon \mathbb{B}^n\times \mathbb{R}^n\to [0,\infty)\) is defined by
    \begin{equation}\label{eq:k:F}
        F(x,y)=\frac{\sqrt{|y|^2-|x|^2|y|^2+\langle x,y\rangle^2}}{\kappa(1-\theta)(1-|x|^2)}-\frac{\theta\langle x,y\rangle}{\kappa(1-\theta) |x|(1-|x|^2)},
    \end{equation}
    for some \(\theta\in(0,1)\) and \(\kappa>0\); while \(\mathsf{m}\) is a measure with density
    \[
        \sigma= \exp(-(n-1)h \rho)\cdot  \sigma_{F},
    \]
    for some \(h\in \mathbb{R}\). As before, \(\rho\) is the distance from the origin and \(\sigma_F\) is the density of the Busemann--Hausdorff measure induced by \(F\). We observe that \(F\) is not a Finsler metric, since it is not continuous at the origin.
    However, for every \(x\ne O\), one has
    \[
        \mathbf{K}(x,x)=-\kappa^2,\quad \overline{\mathbf{S}}(x,x)=(n-1)h,\quad\mbox{and}\quad \mathbf{rev}(x,x)=\frac{1+\theta}{1-\theta}\defeq\lambda=\lambda_F(\mathbb{B}^n).
    \]

    The following result provides an approximation of the above structure by a family of Randers spaces.

    \begin{theorem}\label{thm:k} Let \(\lambda\in(1,\infty)\), \(\theta=\frac{\lambda-1}{\lambda+1}\),  \(\kappa>0\) and \(h\in \mathbb{R}\). Construct the family \(\mathcal{F}_{\kappa,h,\lambda}=\{(\mathbb{B}^n, F_\varepsilon,\mathsf{m}_\varepsilon)\}_{\varepsilon>0}\) of FMMMs, where \(F_\varepsilon\colon T\mathbb{B}^n\simeq\mathbb{B}^n\times\mathbb{R}^n\to [0,\infty)\) is defined by
        \[
            F_\varepsilon(x,y)=\frac{\sqrt{|y|^2-|x|^2|y|^2+\langle x,y\rangle^2}}{k_{\varepsilon}(1-\theta)(1-|x|^2)}
            -\frac{\theta \langle x,y\rangle}{k_{\varepsilon}(1-\theta) |x|(1-|x|^2)}\cdot \frac{\operatorname{arctanh}(|x|)\left(\operatorname{arctanh}(|x|)+2\varepsilon\right)}{\left(\operatorname{arctanh}(|x|)+\varepsilon\right)^2},
        \]
        \(k_{\varepsilon}\) is given by relation~\eqref{eq:k:ke},
        \(\mathsf{m}_\varepsilon\) is a measure with density \(\sigma_\varepsilon=\exp(-(n-1)h \rho_\varepsilon)\cdot \sigma_{F_\varepsilon}\), and \(\sigma_{F_\varepsilon}\) denotes the density of the Busemann--Hausdorff measure induced by \(F_\varepsilon\). The following statements hold:
        \begin{enumerate}[label=(\roman*)]
            \item For every \(\varepsilon>0\), the triple \((\mathbb{B}^n, F_\varepsilon,\mathsf{m}_\varepsilon)\) is a Finsler metric measure manifold. In addition, \(F_\varepsilon\) is a projectively flat Randers metric, with reversibility \(\lambda_{F_\varepsilon}(\mathbb{B}^n)=\lambda\).
            \item For every \(x\in \mathbb{B}^n\) with \(x\ne 0\), one has
                  \begin{equation}\label{eq:k:conditions:limit}
                      \mathbf{K}_\varepsilon(x,x)\to -\kappa^2,\quad \overline{\mathbf{S}}_\varepsilon(x,x)\to (n-1)h,\quad\mbox{and}\quad \mathbf{rev}_\varepsilon(x,x)\to\lambda,\quad\text{as}\quad  \varepsilon\to 0.
                  \end{equation}
            \item For every \(\varepsilon>0\) and \(x\in \mathbb{B}^n\), one has
                  \begin{equation}\label{eq:k:conditions:ineq}
                      \mathbf{K}_\varepsilon(x,x)\le -\kappa^2,\quad
                      \overline{\mathbf{S}}_\varepsilon(x,x)\le (n-1)h,\quad\mbox{and}\quad
                      \mathbf{rev}_\varepsilon(x,x)\le\lambda.
                  \end{equation}
        \end{enumerate}
    \end{theorem}
    \begin{proof}

        For simplicity of presentation, denote \(r=|x|\), \(\overline{r}=\operatorname{arctanh}(|x|)\), and \(\vartheta_\varepsilon=k_\varepsilon(1-\theta)\). Using the standard notation of Randers metrics (see Section~\ref{sec:prelim:randers}) one has
        \[
            F_\varepsilon(x,y)=\sqrt{a_{ij}(x)y^iy^j}+b_i(x)y^i,
        \]
        where
        \[
            a_{ij}(x)=\frac{\delta_{ij}}{\vartheta_{\varepsilon}^2(1-r^2)}+\frac{x^ix^j}{\vartheta_{\varepsilon}^2(1-r^2)^2}
            \quad\mbox{and}\quad
            b_i(x)=-\frac{\theta \overline{r}(\overline{r}+2\varepsilon)x^i}{\vartheta_{\varepsilon} r(1-r^2)(\overline{r}+\varepsilon)^2}.
        \]
        Let \(a^{ij}(x)=\{a_{ij}(x)\}^{-1}\), it follows that
        \[
            \|b\|_a(x)=\sqrt{a^{ij}(x)b_i(x)b_j(x)}=\frac{\theta \overline{r}(\overline{r}+2\varepsilon)}{(\overline{r}+\varepsilon)^2}<\theta<1.
        \]
        Since \(F_\varepsilon\) is also smooth, it is a Randers metric on \(\mathbb{R}^n\). According to Remark~\ref{rem:Randers:flat}, the metric \(F_\varepsilon\) is projectively flat. Indeed, the Klein metric (obtained from \(F_\varepsilon\) for \(\theta=0\)) is projectively flat, and
        \[
            b_i\diff x^i=D\beta(x),\quad\mbox{where}\quad
            \beta(x)=-\frac{\theta}{\vartheta_{\varepsilon}}\left(\overline{r}+\frac{\varepsilon^2}{\overline{r}+\varepsilon}\right).
        \]

        Let \(O\in\mathbb{R}^n\) denote the origin, \(\rho_\varepsilon(x)=d(O,x)\) and \(\overleftarrow{\rho_\varepsilon} (x)=d(x,O)\). Then
        \[
            \rho_\varepsilon(x)=\overline{r}\cdot \frac{\overline{r}(1-\theta)+ \varepsilon }{\vartheta_{\varepsilon}(\overline{r}+\varepsilon)}
            \quad\mbox{and}\quad \overleftarrow{\rho_\varepsilon} (x) =\overline{r}\cdot \frac{\overline{r}(1+\theta)+ \varepsilon }{\vartheta_{\varepsilon}(\overline{r}+\varepsilon)}.
        \]
        In addition, one has
        \[
            D\rho_\varepsilon(x)=\left(\frac{(\overline{r}+\varepsilon)^2-\overline{r}(\overline{r}+2\varepsilon)\theta}{\vartheta_{\varepsilon}r(1-r^2)(\overline{r}+\varepsilon )^2}\right)\cdot x
            \quad\mbox{and}\quad
            \nabla\rho_\varepsilon(x)=\frac{\vartheta_{\varepsilon}(1-r^2)(\overline{r}+\varepsilon)^2}{r((\overline{r}+\varepsilon)^2-\overline{r}(\overline{r}+2\varepsilon)\theta)}\cdot x.
        \]
        According to relation~\eqref{eq:Randers:measure}, the density of the Busemann--Hausdorff measure induced by \(F_\varepsilon\) is
        \[
            \sigma_{F_\varepsilon}=
            \frac{1}{\vartheta_{\varepsilon}^n}\cdot\left(\frac{1}{1-r^2}\right)^\frac{n+1}{2}\cdot\left(1-\frac{\overline{r}^2 (\overline{r}+2 \varepsilon)^2\theta^2}{(\overline{r}+\varepsilon)^4}\right)^\frac{n+1}{2}
            \defeq  \frac{1}{\vartheta_{\varepsilon}^n}\cdot
            \left(\frac{1}{1-\overline{r}^2}\right)^\frac{n+1}{2}
            \cdot f_\varepsilon^\frac{n+1}{2}.
        \]
        Observe that \(1-\theta^2\le f_\varepsilon\le 1\).
        The density of \(\mathsf{m}_\varepsilon\) is
        \[
            \sigma_\varepsilon=   \exp\left(-(n-1)h\rho_\varepsilon\right)\cdot \frac{1}{\vartheta_{\varepsilon}^n}\left(\frac{1}{1-\overline{r}^2}\right)^\frac{n+1}{2}f_\varepsilon^\frac{n+1}{2}.
        \]

        In the sequel, we compute the flag curvature, the reduced \(S\)-curvature, and the reversibility of \((\mathbb{B}^n,F_\varepsilon,\mathsf{m}_\varepsilon)\) along the geodesics starting from \(O\). Recall that \(\vartheta_\varepsilon=k_\varepsilon(1-\theta)\). Concerning the flag curvature, we obtain
        \begin{align*}
            \mathbf{K}_{\rho_\varepsilon}(\overline{r})\defeq\mathbf{K}_\varepsilon(x,x)=
            -\frac{k_{\varepsilon}^2(1-\theta)^2(\overline{r}+\varepsilon )^4
                \left(e_0-e_1\theta +e_2\theta^2
                \right)}{\left((\overline{r}+\varepsilon )^2-\overline{r} (\overline{r}+2 \varepsilon )\theta \right)^4},
        \end{align*}
        where \(e_0=(\overline{r}+\varepsilon)^4\), \(e_1=2\overline{r}^4+8\overline{r}^3\varepsilon+(10\overline{r}^2-3)\varepsilon^2+4\overline{r}\varepsilon^3\), and \(e_2=\overline{r}^4+4\overline{r}^3\varepsilon+(4\overline{r}^2-3)\varepsilon^2\).
        % \[
        %     \begin{cases}
        %         e_0=(\overline{r}+\varepsilon)^4,\\
        %         e_1=2\overline{r}^4+8\overline{r}^3\varepsilon+(10\overline{r}^2-3)\varepsilon^2+4\overline{r}\varepsilon^3,\\
        %         e_2=\overline{r}^4+4\overline{r}^3\varepsilon+(4\overline{r}^2-3)\varepsilon^2.
        %     \end{cases}     
        % \]

        Our goal is to choose \(k_{\varepsilon}\) such that relations~\eqref{eq:k:conditions:limit}~\&~\eqref{eq:k:conditions:ineq} hold.
        For every \(\varepsilon>0\), one has
        \[
            \lim_{\overline{r}\to\infty} \mathbf{K}_{\rho_\varepsilon}(\overline{r})  =-k_{\varepsilon}^2,\quad\text{and}\quad
            \lim_{\overline{r}\to0} \mathbf{K}_{\rho_\varepsilon}(\overline{r})       =-k_{\varepsilon}^2\cdot K_{\varepsilon,0},\quad\text{where}\quad K_{\varepsilon,0}=(1-\theta)^2\left(1+\frac{3\theta(1-\theta)}{\varepsilon^2}\right).
        \]
        By a simple derivative test, we obtain that if \(\varepsilon\le \sqrt{\frac{3(1-\theta)}{8\theta}}\defeq \varepsilon_0\), then \(\overline{r}\mapsto\mathbf{K}_{\rho_\varepsilon}(\overline{r})\) admits a unique maximum on \((0,\infty)\), which is given by
        \[
            -k_{\varepsilon}^2\cdot K_{\varepsilon}, \quad\text{where}\quad K_{\varepsilon}=\frac{6(1-\theta)^2\left((1-\theta ) \left(3 (1-\theta )-2 \varepsilon ^2 \theta \right)+\sqrt{3 (1-\theta)^3 \left(3(1-\theta)-8 \varepsilon ^2 \theta \right)}\right)^3}{\left(3 (1-\theta )^2+\sqrt{3 (1-\theta)^3 \left(3(1-\theta)-8 \varepsilon ^2 \theta \right)}\right)^4}\le 1.
        \]
        For simplicity define  \(K_{\varepsilon}=\infty\), when \(\varepsilon>\varepsilon_0\). To ensure \(\mathbf{K}_{\rho_\varepsilon}(\overline{r})\le-\kappa^2\), \(\forall \varepsilon>0,\overline{r}\ge 0\), choose
        \begin{equation}\label{eq:k:ke}
            k_{\varepsilon}=\max\left\{\kappa,\frac{\kappa}{\sqrt{K_{\varepsilon,0}}},\frac{\kappa}{\sqrt{K_{\varepsilon}}}\right\}.
        \end{equation}
        If \(\varepsilon\to 0\), then \(K_{\varepsilon,0}\to\infty\), \(K_\varepsilon\to1\), and \(k_{\varepsilon}\to\kappa\). Thus, for every \(\overline{r}>0\) we obtain \(\mathbf{K}_{\rho_\varepsilon}(\overline{r})\to-\kappa^2\) as \(\varepsilon\to 0\).

        The reduced \(S\)-curvature along \(\rho_\varepsilon\) can be computed as
        \begin{align*}
            \overline{\mathbf{S}}_{\rho_\varepsilon}(\overline{r})\defeq\overline{\mathbf{S}}(x,x)
             & = (n-1)h-\frac{\vartheta_{\varepsilon}(n+1)(\overline{r}+\varepsilon)\theta \varepsilon^2}{(\overline{r}+\varepsilon)^4-\overline{r}^2(\overline{r}+2\varepsilon)^2\theta^2}\le (n-1)h.
        \end{align*}
        We easily obtain that if \(\overline{r}>0\), then \(\overline{\mathbf{S}}_{\rho_\varepsilon}(\overline{r})\to (n-1)h\) as \(\varepsilon\to 0\). If \(\overline{r}=0\), then \(\overline{\mathbf{S}}_{\rho_\varepsilon}(\overline{r})\to -\infty\) as \(\varepsilon\to 0\). In addition, for every fixed \(\varepsilon>0\) one has \(\overline{\mathbf{S}}_{\rho_\varepsilon}(\overline{r})\to (n-1)h\) as \(\overline{r}\to \infty\).

        Finally, for the reversibility we obtain
        \[
            \mathbf{rev}_{\rho_\varepsilon}(\overline{r})\defeq\mathbf{rev}(x,x)=\frac{1+\frac{\overline{r}(\overline{r}+2\varepsilon)\theta}{(\overline{r}+\varepsilon)^2}}{1-\frac{\overline{r}(\overline{r}+2\varepsilon)\theta}{(\overline{r}+\varepsilon)^2}}\le \frac{1+\theta}{1-\theta}=\lambda.
        \]
        It follows that, if \(\overline{r}>0\), then \(\mathbf{rev}_{\rho_\varepsilon}(\overline{r})\to\lambda\) as \(\varepsilon\to 0\), and for every fixed \(\varepsilon>0\) one has \(\mathbf{rev}_{\rho_\varepsilon}(\overline{r})\to \lambda\) as \(\overline{r}\to \infty\), and
        \(\mathbf{rev}_{\rho_\varepsilon}(\overline{r})\to 1\) as \(\overline{r}\to 0\). The latter limit ensures the continuity of \(F_\varepsilon\) at the origin.
        In addition, \(\lambda_{F_\varepsilon}(\mathbb{B}^n)=\lambda\), for every \(\varepsilon>0\). Indeed, for every \(x\in \mathbb{B}^n\), and \(y\in \mathbb{R}^n\) one has \[\mathbf{rev}_\varepsilon(x,y)=\frac{F_\varepsilon(x,-y)}{F_\varepsilon(x,y)}\le \frac{F_\varepsilon(x,-x)}{F_\varepsilon(x,x)}=\mathbf{rev}_\varepsilon(x,x)\le \lambda,\quad \text{if }x\ne O, \quad\text{and}\quad \mathbf{rev}_\varepsilon(O,y)=1,\]
        which concludes the proof.
    \end{proof}
    We make the following observations.
    \begin{remark}\phantomsection\label{rem:k}
        \begin{enumerate}[label=\textup(\alph*),wide,labelindent=0cm,itemsep=1mm]
            \item\label{rem:k:a} Although there exist various formal similarities between Theorem~\ref{thm:0} and Theorem~\ref{thm:k}, we highlight a main difference, which is the monotonicity of \({r}\mapsto \mathbf{K}_{\rho_\varepsilon}({r})\). In case of non-positive curvature this function is increasing, hence its limit as \({r}\to\infty\) is its upper bound. In case of strong negative curvature, for sufficiently small \(\varepsilon\), it increases to a global maximum \(-k^2_{\varepsilon,\vartheta}\cdot K_\varepsilon\) then decreases (in the proof, that maximum was adjusted to ensure the required upper bound).

                  It is an interesting question, if there exists a perturbation of the metric~\eqref{eq:k:F} (or some other construction), in which \(\mathbf{K}\le-\kappa^2\) along \(\nabla\rho_\varepsilon\) with a strictly increasing flag curvature along \(\nabla\rho_\varepsilon\), or if this phenomena is directly implied by the constraints on the geometry. We note that on each alternative prototypes of families of spaces that were examined during the research, the same phenomena occurred.

            \item As we anticipated in Remark~\ref{rem:0}/\ref{rem:0:b}, the construction of the measure \(\mathsf{m}_\varepsilon\) allows for non-zero upper bound of the reduced \(S\)-curvature, which is particularly important for our applications.

            \item\label{rem:k:c} Since \(K_{\varepsilon,0}\to\infty\) as \(\varepsilon\to 0\), there exists \(\varepsilon_1\in(0,\varepsilon_0)\) such that \(k_\varepsilon=\frac{\kappa}{\sqrt{K_\varepsilon}}\), for every \(\varepsilon\in(0,\varepsilon_1)\). Moreover, since \(k_\varepsilon\to\kappa\) from \emph{above}, there exist \(\overline \varepsilon\in (0,\varepsilon_1)\) and \(\overline \kappa>\kappa\), such that
                  \begin{equation}\label{eq:k:kEstimate}
                      \kappa \le k_{\varepsilon}\le \overline \kappa,\quad\forall \varepsilon\in(0,\overline\varepsilon).
                  \end{equation}

            \item  \label{rem:k:d}Similarly to Remark~\ref{rem:0}/\ref{rem:0:d}, we have \(\max\{F^{*}_\varepsilon(\pm D\rho_\varepsilon)\}=F^{*}_\varepsilon(D\rho_\varepsilon)=1\), since
                  \begin{equation}\label{eq:k:he}
                      F^*_\varepsilon(-D\rho_\varepsilon)=\frac{1-\frac{\overline{r}(\overline{r}+2\varepsilon)}{(\overline{r}+\varepsilon)^2}\theta}{1+\frac{\overline{r}(\overline{r}+2\varepsilon)}{(\overline{r}+\varepsilon)^2}\theta}\defeq h_\varepsilon(\overline{r}),\quad\mbox{where}\quad \frac{1-\theta}{1+\theta}\le h_\varepsilon(\overline{r})\le1.
                  \end{equation}
        \end{enumerate}
    \end{remark}

    \subsection{Sufficient conditions for sharpness}\label{sec:k:conditions}
    In this section, we present sufficient conditions for the sharpness of Hardy inequalities on the spaces of \(\mathcal{F}_{\kappa,h,\lambda}\), introduced in Theorem~\ref{thm:k}.

    Let \((\mathbb{B}^n,F_\varepsilon,\mathsf{m}_\varepsilon)\subset \mathcal{F}_{\kappa,h,\lambda}\) for some \(\varepsilon>0\), \(\lambda\in(1,\infty)\), \(\kappa>0\) and \(h\in \mathbb{R}\). Then for every  \(x\in \mathbb{B}^n\), one has
    \[
        \mathbf{K}_{\varepsilon}(x,x)\le-\kappa^2,\quad \overline{\mathbf{S}}_{\varepsilon}(x,x)\le (n-1)h,\quad\text{and}\quad \mathbf{rev}_{\varepsilon}(x,x)\le \lambda=\lambda_{F_\varepsilon}(\mathbb{B}^n),
    \]
    Denote by \(\rho_\varepsilon\) the distance from the origin.
    By the Laplace comparison theorem (see Theorem~\ref{thm:comp}) one has
    \[
        \Delta_{F_\varepsilon} (\rho_\varepsilon)\ge (n-1)\kappa\coth(\kappa\rho_\varepsilon)-(n-1)h\defeq L^{\kappa,h}(\rho_\varepsilon).
    \]
    Let \(R\in(0,+\infty]\), \(L(t)\le L^{\kappa,h}(t)\) for every \(t\in(0,R)\), and define \(\)
\[
    \Omega_\varepsilon=\{x\in \mathbb{B}^n\mid \rho_\varepsilon(x)<R\}.
\]

Suppose that \((L,W)\) is a \emph{\((p,\rho_\varepsilon,w)\)-Riccati pair} in \((0,R)\), that is \(W\colon(0,R)\to (0,\infty)\) is continuous, \(w\colon(0,R)\to (0,\infty)\) is of class \(C^1\), and there exists \(G\colon(0,R)\to(0,\infty)\) of class \(C^1\) such that
\begin{equation}\label{eq:k:RiccatiODI}
    (G(t)w(t))'+G(t)w(t)L(t)-(p-1)G(t)^{p'}w(t)\geq  W(t)w(t), \quad\forall t\in (0,R).
\end{equation}

By Theorem~\ref{thm:RP}, for every \(u\in C_0^\infty(\Omega_\varepsilon)\) one has
\begin{align}
    \lambda^p\int_{\Omega_\varepsilon} w(\rho_\varepsilon) F^*_\varepsilon(Du)^p\diff \mathsf{m}_\varepsilon
     & \ge\int_{\Omega_\varepsilon} w(\rho_\varepsilon) W(\rho_\varepsilon) |u|^p\diff \mathsf{m}_\varepsilon. \label{eq:k:RP:int:lam}
\end{align}
Suppose that \(G>0\) satisfies the Riccati ODI~\eqref{eq:k:RiccatiODI} with equality, and define \(v\colon (0,R)\to\mathbb{R}\) by
\begin{equation}\label{eq:k:vdef}
    -(\log v(t))'=-\frac{v'(t)}{v(t)}=G(t)^\frac{1}{p-1}.
\end{equation}
Since \(G\) is smooth and positive, \(v\) is smooth, positive, and decreasing. Introduce the notation
\begin{equation}\label{eq:k:I}
    \mathcal{I}^{k,h}(f;a,b)=\int_a^b f(t)\sinh(k t)^{n-1}e^{-(n-1)h t}\diff t,
\end{equation}
where \(k,h\in  \mathbb{R}\) and \(f\colon[a,b]\to\mathbb{R}\) is a continuous function, with \(0<a<b<R\). We have the following result.

% \clearpage
\begin{theorem}\label{thm:k:sharp}
    Fix \(\delta_0>0\) and let \(t_1=t_1(\delta),t_2=t_2(\delta),t_3=t_3(\delta),t_4=t_4(\delta)\) be continuous functions admitting a (possibly infinite) limit at infinity, such that for every \(\delta>\delta_0\) one has \(0<t_1<t_2<t_3<t_4<R\). Let  \(\varepsilon=\varepsilon(\delta)\) such that
    \(\varepsilon\to 0\), as \(\delta\to\infty\), and suppose that the following limits hold true:
    \begin{align*}
        \lim_{\delta\to\infty} l_0^{k_\varepsilon,h}(\delta) & =0,\quad\mbox{where}\quad l_0^{k_\varepsilon,h}(\delta)\defeq \frac{v(t_2)^p\cdot \mathcal{I}^{k_\varepsilon,h}(w;t_1,t_2)}{(t_2-t_1)^p\cdot \mathcal{I}^{k_\varepsilon,h}(wW v^p;t_2,t_3)}+\frac{v(t_3)^p\cdot \mathcal{I}^{k_\varepsilon,h}(w;t_3,t_4)}{(t_4-t_3)^p\cdot \mathcal{I}^{k_\varepsilon,h}(wW v^p;t_2,t_3)}, \\
        \lim_{\delta\to\infty} l_1^{k_\varepsilon,h}(\delta) & =1,\quad\mbox{where}\quad l_1^{k_\varepsilon,h}(\delta)\defeq\max_{t\in [t_2,t_3]} \frac{G(t)^{p'}}{W (t)},\\
        \lim_{\delta\to\infty}l_2^{k_\varepsilon,h}(\delta) &=\frac{1}{\lambda}\quad\text{where}\quad l_2^{k_\varepsilon,h}(\delta)\defeq h_\varepsilon((1-\theta)\kappa t_2),\quad\textup{(see relation~\eqref{eq:k:he})}.
    \end{align*}
    Then inequality~\eqref{eq:k:RP:int:lam} is sharp on \(\mathcal{F}_{\kappa,h,\lambda}\) in the sense that it no longer holds on each FMMM of \(\mathcal{F}_{\kappa,h,\lambda}\) if we multiply its right hand side by a constant \(c>1\).
\end{theorem}
\begin{proof}

    Let \(\varepsilon>0\) and \(T=(t_1,t_2,t_3,t_4)\). As before, define the limit function \(u^\star_\varepsilon \colon \Omega\to (0,\infty)\) of~\eqref{eq:k:RP:int:lam} and its truncation \(u^\star_{\varepsilon,T}\) by
    % Define the limit function  \(u^\star_\varepsilon\colon \Omega\to (0,\infty)\) of inequality~\eqref{eq:k:RP:int:lam} by 
    \[u^\star_\varepsilon(x)=v(\rho_\varepsilon(x)),\quad\mbox{and}\quad u^\star_{\varepsilon,T}(x)=v_{\varepsilon,T}(\rho_\varepsilon(x)),\]
    where
    \begin{equation}\label{eq:k:trunc}
        v_{\varepsilon,T}(t)=\begin{cases}
            0,                           & \mbox{if }0<t\le t_1,      \\
            v(t_2)\frac{t-t_1}{t_2-t_1}, & \mbox{if }t_1\le t\le t_2, \\
            v(t),                        & \mbox{if }t_2\le t\le t_3, \\
            v(t_3)\frac{t_4-t}{t_4-t_3}, & \mbox{if }t_3\le t\le t_4, \\
            0,                           & \mbox{if }t_4\le t<R,
        \end{cases}\quad\mbox{and}\quad v_{\varepsilon,T}'(t)=\begin{cases}
            0,                       & \mbox{if }0<t<t_1,   \\
            \frac{v(t_2)}{t_2-t_1},  & \mbox{if }t_1<t<t_2, \\
            v'(t),                   & \mbox{if }t_2<t<t_3, \\
            -\frac{v(t_3)}{t_4-t_3}, & \mbox{if }t_3<t<t_4, \\
            0,                       & \mbox{if }t_4< t<R.
        \end{cases}
    \end{equation}
    % Note that \(v'_{\varepsilon,T}\) is not defined in \(t_1,t_2,t_3\) and \(t_4\), while \(v_{\varepsilon,T}\) and \(v'_{\varepsilon,T}\) are supported in \([t_1,t_4]\).

    Proceeding similarly to the proof of Theorem~\ref{thm:0:sharp}, we reduce the following integrals to one dimension:
    \[
        I\defeq \frac{1}{n\omega}\int_{\Omega_\varepsilon} w(\rho_\varepsilon) F^*_\varepsilon(Du^\star_{\varepsilon,T})^p\diff \mathsf{m}_\varepsilon\quad\mbox{and}\quad
        J\defeq\frac{1}{n\omega}\int_{\Omega_\varepsilon} w(\rho_\varepsilon) W(\rho_\varepsilon) (u_{\varepsilon,T}^\star)^p\diff \mathsf{m}_\varepsilon.
    \]
    For simplicity, denote \(r=|x|\), \(\overline{r}=\operatorname{arctanh}(r)\), and \(\vartheta_{\varepsilon}=k_\varepsilon(1-\theta)\).  Introduce the following notations:
    \[
        t\defeq\rho_\varepsilon(x)=\overline{r}\cdot \frac{\overline{r}(1-\theta)+ \varepsilon }{\vartheta_{\varepsilon}(\overline{r}+\varepsilon)}%=\operatorname{arctanh}(r)\cdot \frac{\operatorname{arctanh}(r)\cdot(1-\theta)+ \varepsilon }{\vartheta_{\varepsilon}(\operatorname{arctanh}(r)+\varepsilon)}, 
        \quad\text{and}\quad e(t)\defeq \exp(-(n-1)ht).
    \]
    % and denote \[e(t)= \exp(-(n-1)ht).\]
    By a simple computation, we obtain that
    \begin{align*}
        r                                  & =\tanh \varphi_\varepsilon(t),                                                                                               & \text{where}\quad \varphi_\varepsilon(t)  & =\frac{t\vartheta_\varepsilon -\varepsilon+\sqrt{(t\vartheta_\varepsilon -\varepsilon)^2+4t\vartheta_\varepsilon\varepsilon(1-\theta)}}{2(1-\theta)},                                        \\
        \diff r                            & =\frac{1}{(\cosh\varphi_\varepsilon(t))^2}\varphi_\varepsilon'(t)\diff t,                                               & \text{where}\quad \varphi_\varepsilon'(t) & =\frac{k}{2}\left(1+\frac{t\vartheta_\varepsilon-\varepsilon+2\varepsilon(1-\theta)}{\sqrt{(t\vartheta_\varepsilon -\varepsilon)^2+4t\vartheta_\varepsilon\varepsilon(1-\theta)}}\right)   , \\
        \sigma_\varepsilon                 & =  e(t)\cdot \frac{1}{\vartheta_\varepsilon^n} (\cosh\varphi_\varepsilon(t))^{n+1} (f_\varepsilon(t))^\frac{n+1}{2},         &
        \text{where}\quad f_\varepsilon(t) & =1-\frac{\varphi_\varepsilon^2(t) (\varphi_\varepsilon(t)+2 \varepsilon)^2\theta^2}{(\varphi_\varepsilon(t)+\varepsilon)^4}.
    \end{align*}
    It follows that
    \begin{align}
        \sigma_\varepsilon r^{n-1}\diff r
         & =
        e(t)\cdot\varphi'_\varepsilon(t)
        \frac{1}{\vartheta_\varepsilon^n}(f_\varepsilon(t))^\frac{n+1}{2}
        (\sinh \varphi_\varepsilon(t))^{n-1}\diff t
        \defeq e(t)\cdot\psi_\varepsilon(t)\diff t.
        \label{eq:k:dm}
    \end{align}

    In the sequel, we provide suitable bounds for \(\psi_\varepsilon(t)\). First, we observe that
    \[
        (1-\theta)k_\varepsilon t \le \varphi_\varepsilon(t)\le k_\varepsilon t,\quad
        (1-\theta)k_\varepsilon \le \varphi_\varepsilon'(t)\le k_\varepsilon,\quad\mbox{and}\quad
        1-\theta^2 \le f_\varepsilon(t)\le 1,\quad \forall t\ge 0.
    \]
    Next, we show that there exists \(\widetilde{\varepsilon}>0\) and \(\widetilde c>0\) such that
    \begin{equation}\label{eq:k:sinhapprox}
        \widetilde{c}\cdot \sinh(k_\varepsilon t)\le \sinh(\varphi_\varepsilon(t))\le \sinh(k_\varepsilon t),\quad \forall t\ge 0,\varepsilon\in(0,\widetilde\varepsilon).
    \end{equation}
    Since the hyperbolic sine is increasing, we obtain
    \[
        \sinh((1-\theta)k_\varepsilon t)\le \sinh(\varphi_\varepsilon(t))\le \sinh(k_\varepsilon t),\quad \forall t\ge 0, \varepsilon>0,
    \]
    thus, the second inequality of~\eqref{eq:k:sinhapprox} holds. To verify the first inequality, define
    \[
        \psi(k,t)=\frac{\sinh((1-\theta)k t)}{\sinh(kt)},\quad \forall k>0,t\ge 0.
    \]
    Recall from Remark~\ref{rem:k}/\ref{rem:k:c} that there exists \(\overline \varepsilon>0\) and \(\overline\kappa>\kappa\) such that \[ \kappa \le k_{\varepsilon}\le \overline \kappa,\quad \forall\varepsilon\in(0,\overline\varepsilon).\]
    
    On the one hand, observe that \(t\mapsto \psi(\overline \kappa,t)\) is continuous and
    % \[
    %     \lim_{t\to 0}f(\kappa,t)=1-\theta>0,
    % \]
    \(\psi(\overline \kappa,t)\to 1-\theta>0\) as \(t\to 0\),
    which implies that there exists \(\widetilde{c_1}>0\) and \(\widetilde{t}>0\) such that
    \[
        \sinh((1-\theta)\overline \kappa  t)\ge \widetilde{c_1}\cdot \sinh(\overline \kappa  t),\quad \forall t\in[0,\widetilde{t}].
    \]
    On the other hand, since \(k\mapsto \psi(k,t)\) is decreasing, the above estimate holds for every \(k_\varepsilon\in[\kappa,\overline\kappa]\). Hence, 
    \[
        \sinh((1-\theta)k_{\varepsilon}  t)\ge \widetilde{c_1}\cdot \sinh(k_{\varepsilon}  t),\quad \forall t\in[0,\widetilde{t}], \varepsilon\in(0,\overline\varepsilon).
    \]
    
    Now we focus on the case when \(t>\widetilde{t}\). Studying the asymptotic behavior of \(\varphi_\varepsilon\) we obtain 
    \[
        \varphi_\varepsilon(t)\ge t k_\varepsilon-\frac{\varepsilon\theta}{1-\theta}\defeq tk_\varepsilon-s_\varepsilon,\quad \forall t>0, \varepsilon\in(0,\overline\varepsilon).
    \]
    Note that \(k_\varepsilon\to\kappa\) and \(s_\varepsilon\to0\) as \(
    \varepsilon\to 0\). Thus, since \(\widetilde{t}>0\), there exists an \(\widetilde{\varepsilon}\in(0,\overline{\varepsilon})\) such that 
    \begin{equation}\label{eq:k:etildedef}
        \varphi_\varepsilon(\widetilde t)\ge \widetilde{t}k_\varepsilon-s_\varepsilon>s_\varepsilon>0,\quad \forall \varepsilon\in(0,\widetilde \varepsilon).   
    \end{equation}
    Since \(t\mapsto \varphi_\varepsilon(t)\) is increasing, we also have 
    \begin{equation}\label{eq:k:double}
        \varphi_\varepsilon(t)\ge tk_\varepsilon-s_\varepsilon>s_\varepsilon>0,\quad \forall \varepsilon\in(0,\widetilde \varepsilon),t>\widetilde{t}.
    \end{equation}
    It follows that
    \begin{align*}
        \sinh(\varphi_\varepsilon(t)) & \ge \sinh(tk_\varepsilon-s_\varepsilon)=\sinh(tk_\varepsilon)\cosh(s_\varepsilon)-\cosh(tk_\varepsilon)\sinh(s_\varepsilon) \\
                                      & =\sinh(tk_\varepsilon)\cdot\left(\cosh(s_\varepsilon)-\coth(tk_\varepsilon)\sinh(s_\varepsilon)\right).
    \end{align*}
    Since the hyperbolic cotangent is decreasing, using relation~\eqref{eq:k:double}, we obtain that
    \begin{align*}
        \sinh(\varphi_\varepsilon(t)) & \ge \sinh(tk_\varepsilon)\cdot \left(\cosh(s_\varepsilon)-\coth(2s_\varepsilon)\sinh(s_\varepsilon)\right)= \sinh(tk_\varepsilon)\cdot \frac{1}{2\cosh(s_\varepsilon)}
    \end{align*}
    Define \(s_{\widetilde{\varepsilon}}=\frac{\widetilde{\varepsilon}\theta}{1-\theta}\), the monotonicity of the hyperbolic cosine implies
    \begin{align*}
        \sinh(\varphi_\varepsilon(t)) & \ge \sinh(tk_\varepsilon)\cdot \frac{1}{2\cosh(s_{\widetilde{\varepsilon}})}\defeq \widetilde{c_2}\cdot \sinh(tk_\varepsilon).
    \end{align*}
    Thus, \(\widetilde{c}=\min\{\widetilde{c_1},\widetilde{c_2}\}\) and \(\widetilde{\varepsilon}\) yield inequality~\eqref{eq:k:sinhapprox}, for every \(t\ge 0\).

    Combining the above estimates, we obtain for every \(t\ge 0\) and \(0<\varepsilon<\widetilde{\varepsilon}\) that
    \begin{equation}\label{eq:k:cdef}
        c \sinh(k_\varepsilon t)^{n-1}\defeq \frac{(1-\theta^2)^\frac{n+1}{2}(\widetilde{c})^{n-1}}{\vartheta_{\varepsilon}^{n-1}}\sinh(k_\varepsilon t)^{n-1} \le \psi_\varepsilon(t)\le \frac{k_\varepsilon}{\vartheta^n}\sinh(k_\varepsilon t)^{n-1}\defeq C  \sinh(k_\varepsilon t)^{n-1}.
    \end{equation}

    Since \(v\) is positive and decreasing, the positive homogeneity of \(F^*_\varepsilon\) and the relations from~\eqref{eq:k:trunc} imply
    \begin{equation}\label{eq:k:F:trunc}
        F^*_\varepsilon(Du^\star_{\varepsilon,T})=F^*_\varepsilon(v_{\varepsilon,T}'(\rho_\varepsilon)D\rho_\varepsilon)=\begin{cases}
            \frac{v(t_2)}{t_2-t_1}F^*_\varepsilon(D\rho_\varepsilon), & \text{if } \rho_\varepsilon\in(t_1,t_2),         \\
            |v'(\rho_\varepsilon)|F^*_\varepsilon(-D\rho_\varepsilon) & \text{if } \rho_\varepsilon\in(t_2,t_3),         \\
            \frac{v(t_3)}{t_4-t_3}F^*_\varepsilon(-D\rho_\varepsilon) & \text{if } \rho_\varepsilon\in(t_3,t_4),         \\
            0,                                                        & \text{if }\rho_\varepsilon\in(0,t_1)\cup(t_4,R).
        \end{cases}
    \end{equation}
    According to Remark~\ref{rem:k}/\ref{rem:k:d}, one has
    \[F^*_\varepsilon(D\rho_\varepsilon)=1,\quad\mbox{and}\quad F^*_\varepsilon(-D\rho_\varepsilon)=\frac{1-\frac{\overline{r}(\overline{r}+2\varepsilon)}{(\overline{r}+\varepsilon)^2}\theta}{1+\frac{\overline{r}(\overline{r}+2\varepsilon)}{(\overline{r}+\varepsilon)^2}\theta}\defeq h_\varepsilon(\overline{r}),\quad\mbox{where}\quad \frac{1-\theta}{1+\theta}\le h_\varepsilon(\overline{r})\le1.\]
    Observe that \(h_\varepsilon(\overline{r})\) is decreasing in \( \overline{r}\), increasing in \(\varepsilon\), and approaches to \(\frac{1-\theta}{1+\theta}=\frac{1}{\lambda}\) as \(\varepsilon\to 0\). Clearly one has 
    \[\overline r=\varphi_\varepsilon(t)\ge (1-\theta)k_\varepsilon t\ge(1-\theta)\kappa t,\] thus, we also obtain
    \begin{equation}\label{eq:k:hred}
        h_\varepsilon(\overline{r})=h_\varepsilon(\varphi_\varepsilon(t))\le h_\varepsilon((1-\theta)\kappa t)\le 1.
    \end{equation}
    As before, denote by \(\omega\) the volume of the Euclidean unit ball of dimension \(n\).

    On the one hand, relations~\eqref{eq:k:dm}~and~\eqref{eq:k:F:trunc} yield
    \begin{align*}
        I
         & =\frac{v(t_2)^p}{(t_2-t_1)^p}\int_{t_1}^{t_2} w(t) \psi_\varepsilon(t)e(t)\diff t  + \int_{t_2}^{t_3} w(t)|v'(t)|^p h_\varepsilon^p(\varphi_\varepsilon(t)) \psi_\varepsilon(t)e(t)\diff t
        \\
         & \qquad +\frac{v(t_3)^p}{(t_4-t_3)^p}\int_{t_3}^{t_4} w(t) h_\varepsilon^p(\varphi_\varepsilon(t))\psi_\varepsilon(t)e(t)\diff t.
    \end{align*}
    Applying relations~\eqref{eq:k:cdef},~\eqref{eq:k:hred}, and using the definition of \(\mathcal{I}^{k,h}\) from~\eqref{eq:k:I}, we obtain
    \begin{align*}
        I
         & \le \frac{v(t_2)^p}{(t_2-t_1)^p}\int_{t_1}^{t_2} w(t) \psi_\varepsilon(t)e(t)\diff t+ \int_{t_2}^{t_3} w(t)|v'(t)|^p h_\varepsilon^p((1-\theta)\kappa t) \psi_\varepsilon(t)e(t)\diff t                                                         \\
         & \qquad +\frac{v(t_3)^p}{(t_4-t_3)^p}\int_{t_3}^{t_4} w(t) h_\varepsilon^p((1-\theta)\kappa t)\psi_\varepsilon(t)e(t)\diff t
        \\&\le \frac{Cv(t_2)^p}{(t_2-t_1)^p}\int_{t_1}^{t_2} w(t) \sinh(k_\varepsilon t)^{n-1}e(t)\diff t+ h_\varepsilon^p((1-\theta)\kappa t_2)\int_{t_2}^{t_3} w(t)|v'(t)|^p \psi_\varepsilon(t)e(t)\diff t \\
         & \qquad +\frac{C h_\varepsilon^p((1-\theta)\kappa t_3)v(t_3)^p}{(t_4-t_3)^p}\int_{t_3}^{t_4} w(t)\sinh(k_\varepsilon t)^{n-1}e(t)\diff t                                                                                                                \\
         & \le \frac{Cv(t_2)^p\cdot \mathcal{I}^{k_\varepsilon,h}(w;t_1,t_2)}{(t_2-t_1)^p}+ h_\varepsilon^p((1-\theta)\kappa t_2)\int_{t_2}^{t_3} w(t)|v'(t)|^p \psi_\varepsilon(t)e(t)\diff t +\frac{C v(t_3)^p\cdot\mathcal{I}^{k_\varepsilon,h}(w;t_3,t_4)}{(t_4-t_3)^p}.
    \end{align*}
   
    By relation~\eqref{eq:k:vdef}, for every \(t\in[t_2,t_3]\) one has
    \[
        \frac{1}{W(t)}\left|\frac{v'(t)}{v(t)}\right|^p=\frac{1}{W(t)}\left(-\frac{v'(t)}{v(t)}\right)^p=\frac{G(t)^{p'}}{W(t)}.
    \]
    The above function is continuous, thus, it attains it maximum on \([t_2,t_3]\), hence we have
    \begin{align*}
         \int_{t_2}^{t_3} w(t)|v'(t)|^p \psi_\varepsilon(t)\diff t&=\int_{t_2}^{t_3} \frac{1}{W(t)}\left|\frac{v'(t)}{v(t)}\right|^p w(t) W(t)(v(t))^p \psi_\varepsilon(t)\diff t\\
         &\le  \max_{t\in[t_2,t_3]} \frac{G(t)^{p'}}{W(t)}\cdot \int_{t_2}^{t_3} w(t) W(t) (v(t))^p \psi_\varepsilon(t)e(t)\diff t.
    \end{align*}

         On the other hand, we obtain
    \begin{align}\label{eq:k:estimates}
        J=\int_{0}^{R} w(t) W(t) (v_{\varepsilon,T}(t))^p \psi_\varepsilon(t)e(t)\diff t
         & \ge \int_{t_2}^{t_3} w(t) W(t) (v(t))^p \psi_\varepsilon(t)e(t)\diff t\ge c\cdot \mathcal{I}^{k_\varepsilon,h}(wW v^p;t_2,t_3).
    \end{align}
    Combining the above results, and using both inequalities of~\eqref{eq:k:estimates}, it follows that
    \begin{align*}
        \frac{\lambda^p I}{J}
         & \le \lambda^p h_\varepsilon^p((1-\theta)\kappa t_2)\cdot  \max_{t\in[t_2,t_3]} \frac{G(t)^{p'}}{W(t)}                                                                                                                                                                              \\
         & \qquad +\frac{\lambda^p C}{c}\left(\frac{v(t_2)^p\cdot \mathcal{I}^{k_\varepsilon,h}(w;t_1,t_2)}{(t_2-t_1)^p\cdot \mathcal{I}^{k_\varepsilon,h}(wW v^p;t_2,t_3)}+\frac{v(t_3)^p\cdot \mathcal{I}^{k_\varepsilon,h}(w;t_3,t_4)}{(t_4-t_3)^p\cdot \mathcal{I}^{k_\varepsilon,h}(wW v^p;t_2,t_3)}\right) \\
         & =\lambda^p\cdot (l_2^{k_\varepsilon,h}(\delta))^p\cdot l_1^{k_\varepsilon,h}(\delta)+\frac{\lambda^p C}{c}\cdot l_0^{k_\varepsilon,h}(\delta).
    \end{align*}
    By our assumptions, the last relation approaches \(1\) as \(\delta\to\infty\), thus, the inequality~\eqref{eq:k:RP:int:lam} is sharp on \(\mathcal{F}_{\kappa,h,\lambda}\).
\end{proof}
Observe that in contrast with Theorem~\ref{thm:0:sharp}, the limits appearing in the statement of the above result are strongly dependent on \(\varepsilon\).  The reason for this stems from the observation of Remark~\ref{rem:k}/\ref{rem:k:a}: Since the flag curvature is not increasing along \(\nabla\rho\), we need to work with \(k_\varepsilon\) instead of \(\kappa\), which explains the dependence on \(\varepsilon\). Fortunately, by suitable choices of \(\varepsilon(\delta)\), one can provide sharpness results in this setting, as well. Such an example is presented in the sequel.

\subsection{The sharpness of a McKean-type spectral gap estimate}\label{sec:k:applications}
In this section, we prove a sharpness result concerning a McKean-type spectral gap estimate on \(\mathcal{F}_{\kappa,h,\lambda}\) (which for \(p=2\) reduces to Theorem~\ref{thm:intro:McKean:sharp}). Our last result can be stated as follows.

\begin{theorem}\label{thm:k:app:sharp}  Let \(n\ge 2\), \(\lambda,p\in (1,\infty)\), and \(\kappa>h\ge0\). For every \((\mathbb{B}^n,F_\varepsilon,\mathsf{m}_\varepsilon)\in \mathcal{F}_{\kappa,h,\lambda}\) and \(u\in C_0^\infty(\mathbb{B}^n)\) one has
    \begin{equation}\label{eq:k:app:McKean}
        \lambda^p\displaystyle\int_{\mathbb{B}^n} F_\varepsilon^*(Du)^p\diff \mathsf{m}_\varepsilon\ge c_{n,p,\kappa,h}\int_{\mathbb{B}^n} |u|^p \diff \mathsf{m}_\varepsilon,\quad\text{where}\quad c_{n,p,\kappa,h}=\left(\frac{(n-1)(\kappa-h)}{p}\right)^p.
    \end{equation}
    Moreover, \(c_{n,p,\kappa,h}\) is \emph{sharp} in \(\mathcal{F}_{\kappa,h,\lambda}\), in the sense that it is the greatest constant with this property.
\end{theorem}
\begin{proof}
    The validity of inequality~\eqref{eq:k:app:McKean} follows from Theorem~\ref{thm:McKean:genproof}. To prove the sharpness, we use Theorem~\ref{thm:k:sharp}. Recall the definition of \(\widetilde\varepsilon\) from~\eqref{eq:k:etildedef}. Let \(\delta>1\), such that \(\varepsilon=\frac{1}{\delta}\in(0,\widetilde\varepsilon)\), and define
    \[
        t_1(\delta)=\delta,\quad
        t_2(\delta)=2\delta,\quad
        t_3(\delta)=3\delta,\quad\mbox{and}\quad
        t_4(\delta)=4\delta.
    \]
    By a simple computation, we obtain \[l_2^{k_\varepsilon,h}=h_\varepsilon((1-\theta)\kappa t_2)=\frac{1+4\delta^2(1-\theta)^2\kappa+4\delta^4(1-\theta)^3\kappa^2}{1+4\delta^2(1-\theta^2)\kappa+4\delta^4(1-\theta)^2(1+\theta)\kappa^2}\to \frac{1-\theta}{1+\theta}=\lambda,\quad\text{as}\quad \delta\to \infty.\]
    Consider the following functions (see also the proof of Theorem~\ref{thm:McKean:genproof}): 
    \[
        w(t)= 1,
        \quad W(t)=c_{n,p,\kappa,h},
        \quad G(t)\equiv (c_{n,p,\kappa,h})^\frac{p-1}{p},\quad\mbox{and}\quad v(t)^p=e^{-(n-1)t (\kappa-h)},\quad \forall t>0.
    \]
    Since \(k_\varepsilon>\kappa>h\), the following expressions are  strictly increasing in \(t\):
    \begin{align*}
        w(t)\sinh(k_\varepsilon t)^{n-1}e^{-(n-1)h t}           & = \sinh(k_\varepsilon t)^{n-1}e^{-(n-1)h t},                     \\
        w(t)W(t)v(t)^p\sinh(k_\varepsilon t)^{n-1}e^{-(n-1)h t} & =c_{n,p,\kappa,h}\sinh(k_\varepsilon t)^{n-1}e^{-(n-1)\kappa t}.
    \end{align*}
    On the one hand, we have 
    \begin{align*}
        \frac{v(t_2)^p\cdot\mathcal{I}^{k_\varepsilon,h}(w;t_1,t_2)}{(t_2-t_1)^p\cdot \mathcal{I}^{k_\varepsilon,h}(wWv^p;t_2,t_3)}
         & \le \frac{v(t_2)^p}{(t_2-t_1)^p}\cdot \frac{(t_2-t_1)\sinh(k_\varepsilon t_1)^{n-1} e^{-(n-1)ht_1}}{(t_3-t_2)c_{n,p,\kappa,h}\sinh(k_\varepsilon t_3)^{n-1}e^{-(n-1)h t_3}} \\
         & =\frac{e^{(n-1)\delta(\kappa+h)}\sinh(\delta k_\varepsilon)^{n-1}}{\delta^pc_{n,p,\kappa,h} \sinh(3\delta k_\varepsilon)^{n-1}}\\
         &= \frac{e^{(n-1)\delta(\kappa+h-2k_\varepsilon)}}{\delta^pc_{n,p,\kappa,h}}\cdot \frac{e^{(n-1)\delta(2k_\varepsilon)}\sinh(\delta k_\varepsilon)^{n-1}}{\sinh(3\delta k_\varepsilon)^{n-1}}\to 0\cdot 1=0, \quad\text{as}\quad \delta\to\infty.
    \end{align*}
    On the other hand, observe that
    \begin{align*}
        \frac{v(t_3)^p\cdot\mathcal{I}^{k_\varepsilon,h}(w;t_3,t_4)}{(t_4-t_3)^p\cdot \mathcal{I}^{k_\varepsilon,h}(wWv^p;t_2,t_3)}
        &\le \frac{v(t_3)^p}{(t_4-t_3)^p}\cdot \frac{(t_4-t_3)\sinh(\kappa t_3)^{n-1} e^{-(n-1)ht_3}}{(t_3-t_2)c_{n,p,\kappa,h}\sinh(\kappa t_3)^{n-1}e^{-(n-1)\kappa t_3}} \\
        &=\frac{1}{\delta^pc_{n,p,\kappa,h} }\to 0,\quad\mbox{as }\delta\to\infty.
    \end{align*}
    Combining the above two limits, we obtain \(l_1^{k_\varepsilon,h}(\delta)\to 0\) as \(\delta\to\infty\).

    Finally, since 
    \[
        \frac{G(t)^{p'}}{W(t)}=1, \quad\forall t>0,
    \]
    we also obtain \(l_1^{k_\varepsilon,h}(\delta)=1\).
    Thus, Theorem~\ref{thm:k:sharp} yields the sharpness of the inequality~\eqref{eq:k:app:McKean}, which concludes the paper.
\end{proof}

\noindent\textbf{Acknowledgement.} The author would like to thank prof.~Alexandru Kristály for the valuable discussions on the subject of the paper. He is also grateful to the anonymous reviewer for carefully reading the article and providing valuable comments.

% \bibliographystyle{plain}
% \bibliography{references}

\begin{thebibliography}{10}

\bibitem{balinsky2015analysis}
A.~A. Balinsky, W.~D. Evans, and R.~T. Lewis.
\newblock {\em {The analysis and geometry of Hardy's inequality}}.
\newblock Springer Verlag Universitext, 2015.

\bibitem{bao2000introduction}
D.~Bao, S.-S. Chern, and Z.~Shen.
\newblock {\em {An introduction to Riemann-Finsler geometry}}, volume 200.
\newblock Springer Science \& Business Media, 2000.

\bibitem{ghoussoub2013functional}
N.~Ghoussoub and A.~Moradifam.
\newblock {\em {Functional Inequalities: New Perspectives and New Applications: New Perspectives and New Applications}}, volume 187 of {\em Mathematical Surveys and Monographs}.
\newblock AMS, 2013.

\bibitem{hardy1920note}
G.~H. Hardy.
\newblock {Note on a theorem of Hilbert}.
\newblock {\em Math. Z.}, 6(3-4):314--317, 1920.

\bibitem{huang2020sharp}
L.~Huang, A.~Kristály, and W.~Zhao.
\newblock {Sharp uncertainty principles on general Finsler manifolds}.
\newblock {\em Trans. Am. Math. Soc.}, 373(11):8127--8161, 2020.

\bibitem{riccatipair2023}
S.~Kajántó, A.~Krist{\'a}ly, I.~R. Peter, and W.~Zhao.
\newblock {A generic functional inequality and Riccati pairs: an alternative approach to Hardy-type inequalities}.
\newblock {\em Math. Ann}, 390:3621--3663, 2024.

\bibitem{kristaly2024influence}
A~Kristály, B.~Li, and W.~Zhao.
\newblock {Failure of famous functional inequalities on Finsler manifolds: The influence of \(S\)-curvature}.
\newblock {\em arXiv preprint: 2409.05497}, 2024.

\bibitem{mckean1970upper}
H.~McKean.
\newblock {An upper bound to the spectrum of $\Delta $ on a manifold of negative curvature}.
\newblock {\em J. Differ. Geom.}, 4(3):359--366, 1970.

\bibitem{ohta2009heat}
S.~Ohta and K.-T. Sturm.
\newblock {Heat flow on Finsler manifolds}.
\newblock {\em Commun. Pure Appl. Math.}, 62(10):1386--1433, 2009.

\bibitem{rademacher2004sphere}
H.~Rademacher.
\newblock {A sphere theorem for non-reversible Finsler metrics}.
\newblock {\em Math. Ann.}, 328(3):373--387, 2004.

\bibitem{rapcsak}
A~Rapcsák.
\newblock {Uber die bahntreuen Abbildungen metrisher R\"aume}.
\newblock {\em Publ. Math. Debr.}, 8:285--290, 1961.

\bibitem{shen2001lectures}
Z.~Shen.
\newblock {\em {Lectures on Finsler geometry}}.
\newblock World Scientific, 2001.

\bibitem{wu2007comparison}
B.~Y. Wu and Y.~L. Xin.
\newblock {Comparison theorems in Finsler geometry and their applications}.
\newblock {\em Math. Ann.}, 337(1):177--196, 2007.

\bibitem{yin2014first}
S.-T. Yin and Q.~He.
\newblock {The first eigenvalue of Finsler p-Laplacian}.
\newblock {\em Differ. Geom. Appl.}, 35:30--49, 2014.

\bibitem{Z}
W.~Zhao.
\newblock {Hardy inequalities with best constants on Finsler metric measure manifolds}.
\newblock {\em J. Geom. Anal.}, 31(2):1992--2032, 2021.

\end{thebibliography}

\end{document}